\documentclass[a4paper,11pt]{amsart}
\usepackage[utf8]{inputenc}
\usepackage{csquotes}
\usepackage{amsmath, amssymb, amsthm, xcolor,enumerate, enumitem}
\usepackage{tikz-cd}
\usepackage[all]{xy}
\usepackage{microtype}
\usepackage[colorlinks=true, citecolor=blue, linkcolor=blue, bookmarks=true]{hyperref}
\usepackage{bbm}
\usepackage{graphicx}
\usepackage{adjustbox}
\usepackage{comment}
\usepackage[capitalize]{cleveref}
\usepackage{stmaryrd}
\usepackage{mathtools}

\allowdisplaybreaks
\newtheorem{theorem}{Theorem}[section]
\newtheorem{lemma}[theorem]{Lemma}
\newtheorem{corollary}[theorem]{Corollary}
\newtheorem{proposition}[theorem]{Proposition}
\newtheorem*{theorem*}{Theorem}
\numberwithin{equation}{section}

\theoremstyle{definition}
\newtheorem{definition}[theorem]{Definition}
\newtheorem{notation}[theorem]{Notation}
\newtheorem{remark}[theorem]{Remark}

\numberwithin{equation}{section}

\newcommand{\eins}[0]{\mathbf{1}}
\newcommand{\ti}{\tilde}

\newcommand{\wv}{\widetilde{\mathbbm{v}}}
\newcommand{\wf}{\widetilde{\phi}}
\newcommand{\cF}{\mathcal{F}}
\newcommand{\cB}{\mathcal{B}}

\newcommand{\cG}{\mathcal{G}}
\newcommand{\cH}{\mathcal{H}}
\newcommand{\cO}{\mathcal{O}}
\newcommand{\cU}{\mathcal{U}}
\newcommand{\cV}{\mathcal{V}}
\newcommand{\cR}{\mathcal{R}}
\newcommand{\cM}{\mathcal{M}}
\newcommand{\cC}{\mathcal{C}}

\newcommand{\Z}{\mathcal{Z}}

\newcommand{\bK}{\mathbb{K}}

\newcommand{\cZ}{\mathcal{Z}}
\newcommand{\bT}{\mathbb{T}}
\newcommand{\bC}{\mathbb{C}}
\newcommand{\bN}{\mathbb{N}}

\newcommand{\bZ}{\mathbb{Z}}
\newcommand{\bu}{\mathbbm{u}}
\newcommand{\bs}{\mathbbm{s}}
\newcommand{\fu}{\mathfrak{u}}
\newcommand{\fv}{\mathfrak{v}}
\newcommand{\bv}{\mathbbm{v}}
\newcommand{\bw}{\mathbbm{w}}

\newcommand{\Hilb}{\mathrm{Hilb}}

\newcommand{\ob}{\mathrm{ob}}
\newcommand{\Ad}{\mathrm{Ad}}
\newcommand{\Aut}{\mathrm{Aut}}
\newcommand{\Inn}{\mathrm{Inn}}
\newcommand{\Out}{\mathrm{Out}}

\newcommand{\im}{\operatorname{im}}
\newcommand{\id}{\operatorname{id}}

\newcommand{\Hom}{\operatorname{Hom}}

\newcommand{\coker}{\operatorname{coker}}

\newcommand{\cstar}[0]{C$^*$}

\title[$G$-kernels with the Rokhlin property]{Classification of anomalous actions of finite groups with the Rokhlin property}

\author{Sergio Girón Pacheco}
\author{Gábor Szabó}
\address{\hskip-\parindent Department of Mathematics, KU Leuven, Celestijnenlaan 200B, box 2400, B-3001 Leuven, Belgium.}
\email{sergio.gironpacheco@kuleuven.be}
\email{gabor.szabo@kuleuven.be}

\thanks{GS funded by the European Union. Views and opinions expressed are however those of the authors only and do not necessarily reflect those of the European Union or the European Research Council. Neither the EU nor the ERC can be held responsible for them.}

\begin{document}

\begin{abstract}
Given a finite group $G$, we develop a generalization of the fundamental classification of Rokhlin $G$-actions on \cstar-algebras to the setting of anomalous $G$-actions with the Rokhlin property.
For Rokhlin $G$-kernels on \cstar-algebras covered by the classification program, this implies that the induced $G$-action on some known invariants determines the conjugacy class.
Extending a result of Izumi, we show that $G$-kernels with the Rokhlin property on Kirchberg algebras are classified by their anomaly and the induced module structure on the K-theory groups.
We explore K-theoretic obstructions for the existence of Rokhlin $G$-kernels on general separable C$^*$-algebras and use these to characterise which $G$-module structures arise as K-groups of Kirchberg algebras admitting a Rokhlin $G$-kernel with a given lifting obstruction.
\end{abstract}

\maketitle

\setcounter{tocdepth}{1}
\tableofcontents

\section*{Introduction}
\renewcommand{\thetheorem}{\Alph{theorem}}
\setcounter{theorem}{0}

The classification of symmetries of simple operator algebras has attracted great interest since the 80's.
The first major results were obtained in the context of group actions on the hyperfinite II$_1$ factor $\cR$.
Connes, Jones, and Ocneanu proved that every countable discrete amenable group admits a unique pointwise outer action on $\cR$ \cite{CO75,CO77,JO80,OC85}.
These results were subsequently extended to a complete classification of group actions on injective factors by Sutherland, Katayama, and Takesaki \cite{KATASU98}.

In contrast, partly due to the intricate interplay between symmetries and K-theoretic invariants, a comparably satisfactory classification theory for group actions on simple amenable C$^*$-algebras is developing at much slower pace. However, in recent years, there has been significant progress, with a variety of substantial classification results for such actions, making use of KK-theoretic, and homotopy-theoretic invariants \cite{IZMA21,IZ10,GASZ24}. 

A major milestone in the classification of group actions on C$^*$-algebras was achieved by Izumi through his classification of finite group actions with the Rokhlin property \cite{IZU04I,IZU04II}.
Subsequently these ideas could even be transferred to actions of compact (quantum) groups on \cstar-algebras \cite{AranoKubota17, Gardella22, Gardella25, BarlakSzaboVoigt17}.
Building on earlier work of Hermann and Jones \cite{HEJO82,HEJO83}, Izumi introduced the Rokhlin property and obtained a classification of finite group actions with this property on unital \cstar-algebras based on remarkably elementary methods, which were later refined and extended to the nonunital case \cite{GASA16}.
Izumi applied his classification machinery to Kirchberg algebras, where the conjugacy class of a Rokhlin action is entirely encoded in the induced module structure on the K-theory groups.
Even though the Rokhlin property imposes strong constraints for the possible module structures, Izumi also obtained a range theorem, giving a complete group cohomological characterization of the $G$-module structure that are induced from Rokhlin $G$-actions on the K-groups of a given Kirchberg algebra for a finite group $G$.

Already at the time of the Connes--Jones--Ocneanu classification theorems, a great amount of interest was attracted by (the classification of) symmetries beyond just group actions.
In these papers the authors considered \emph{$G$-kernels}. For a discrete group $G$, a $G$-kernel on a C$^*$-algebra $A$ is a group homomorphism
\[
\alpha:G\rightarrow \Out(A)=\Aut(A)/\Inn(A).
\]
Any $G$-kernel $\alpha$ gives rise to a class $\omega\in H^3(G,Z\cU\cM(A))$ called the \emph{lifting obstruction} of $\alpha$, which is precisely the obstruction to lifting $\alpha$ to a coycle action on $A$.
Connes, Jones and Ocneanu show that for any countable amenable group $G$, faithful $G$-kernels on $\cR$ are classified up to conjugacy by their lifting invariant (see also \cite{KATA03,KATA07,KATA09} for the case of other injective factors).
Moreover, for any countable, discrete group $G$ any class in $H^3(G,\bT)$ arises as a lifting obstruction of a $G$-kernel on $\cR$ and hence on any injective factor \cite{JON79}. 

In contrast, the classification theory for $G$-kernels on C$^*$-algebra remains in its infancy.
Up to now there are only sporadic results \cite{IZ23, GP25, CHHPJO24} with strong assumptions on the structure of the underlying C$^*$-algebra.
In this paper we initiate the general classification of finite group $G$-kernels with the Rokhlin property, and apply it on simple amenable C$^*$-algebras.
To demonstrate the latter, we obtain the full-fledged analogue of a result of Izumi that gave a full description of Rokhlin $G$-actions on Kirchberg algebras for finite $G$; see Corollary \ref{cor:Gkerclass} in the main body of the paper.

\begin{theorem} \label{introthrm:Gkerclass}
Let $A$ and $B$ be (unital) Kirchberg algebras satisfying the UCT and $G$ be a finite group.
Let $\alpha:G\rightarrow \Out(A)$ and $\beta:G\rightarrow \Out(B)$ be two $G$-kernels with the Rokhlin property.
Then $\alpha$ and $\beta$ are conjugate if and only if $\alpha$ and $\beta$ have the same lifting obstructions in $H^3(G,\bT)$ and the $G$-actions $K_*(\alpha): G\curvearrowright K_*(A)$ and $K_*(\beta): G\curvearrowright K_*(B)$ are (unitally) conjugate.
\end{theorem}

Theorem \ref{introthrm:Gkerclass} significantly generalises earlier work of the first-named author \cite{GP25}, where a preliminary classification result was obtained under the assumption that $A$ and $B$ absorb the UHF-algebra $M_{|G|^\infty}$ tensorially.
We note that we can in fact obtain Theorem \ref{introthrm:Gkerclass} in the setting of finite classifiable \cstar-algebras of real rank zero as well. 

The study of $G$-kernels on C$^*$-algebras has gained more attention in recent years \cite{JO21,EVGI23,THESIS,GP25,IZ23, ARKIKU23, KI24,GPIZPE25}.
Unlike the case of injective factors, it was shown by Evington and the first-named author \cite{EVGI23} that there are K-theoretic restrictions for the values of the lifting obstruction for $G$-kernels on C$^*$-algebras.
For example, no non-trivial lifting obstruction can occur on the Jiang--Su algebra $\cZ$, and for UHF-algebras the order of the lifting obstruction is constrained by K-theory.
Further restrictions, also applicable to Kirchberg algebras, were established by Izumi \cite{IZ23}. 

In the present paper we show that, in the presence of the Rokhlin property and a finite group $G$, there are no additional restrictions on the lifting obstruction of a $G$-kernel on a Kirchberg algebra, beyond those already implied by the existence of a Rokhlin $G$-action known by Izumi's work.
For ease of exposition, we state this result below in the setting of stable Kirchberg algebras; see Theorem \ref{thm:rangeofinvstable} for the more rigorous and notation-heavy statement.

\begin{theorem} \label{introthm:rangeofinvstable}
Let $G$ be a finite group, $[\omega]\in H^3(G,\bT)$ and $M_0, M_1$ a pair of countable completely cohomologically trivial $G$-modules.
Then there exists a Rokhlin $G$-kernel on a stable UCT Kirchberg algebra $A$ with lifting invariant $[\omega]$ and such that $K_*(A)\cong M_*$ as $G$-modules.    
\end{theorem}
Complete cohomological triviality is a strong cohomological constraint on a $G$-module introduced by Izumi in \cite{IZU04II} as a tool to characterise the K-theoretic $G$-modules that are induced by Rokhlin $G$-actions.
An important structure theorem in this endevour is \cite[Theorem 3.3]{IZU04II}, where it is shown that if $\alpha$ is a Rokhlin $G$-action on a unital, simple, separable C$^*$-algebra $A$, then the $G$-module structures induced by $K_*(\alpha)$ are completely cohomologically trivial. 
A key ingredient towards Theorem \ref{introthm:rangeofinvstable} is the analogous result in the case of $G$-kernels.
Some of the conceptual shortcuts appealed to in \cite{IZU04II} do however not apply in the case of $G$-kernels.
Consequently, our proof requires a more direct methodological approach than that of \cite{IZU04II}.
In the process, this allows us to explore the aforentioned K-theoretic obstructions for Rokhlin $G$-kernels on arbitrary separable \cstar-algebras; see Theorem \ref{thm:generalcohomologyvanishing}.

To achieve Theorem \ref{introthrm:Gkerclass}, we draw inspiration from the Elliott classification programme, proving existence and uniqueness theorems for morphisms between lifts of Rokhlin $G$-kernels (often called \emph{anomalous $G$-actions}). This is encouraged by the strategy proposed by the second named author for genuine group actions in \cite{SZ21} and analogous to the methods by Gardella--Santiago \cite{GASA16}. An \emph{anomalous action} of a discrete group $G$ consists of a pair $(\alpha,\fu)$, where $\alpha:G\rightarrow \Aut(A)$ and $\fu:G\times G\rightarrow \cU\cM(A)$ are maps such that
\[
\alpha_g \alpha_h=\Ad(\fu_{g,h})\alpha_{gh}.
\]
and the resulting 3-cocycle $[(g,h,k)\mapsto \alpha_g(\fu_{h,k})\fu_{g,hk}\fu_{gh,k}^*\fu_{g,h}^*\in Z\cU\cM(A)]$ is scalar-valued.\footnote{Anomalous actions and lifts of $G$-kernels only coincide when $A$ has trivial centre, which we shall assume for the rest of the discussion. We direct the interested reader to Section \ref{subsec:anomaction} for details.}
Due to the fact that the unitary cocycles $\fu$ arising from $G$-kernels with non-trivial lifting invariants are necessarily multiplier-valued, the obvious generalisation of cocycle morphisms between anomalous actions (Definition \ref{def:cocyclemorphism}) is not suitable for our purposes.
Instead, we consider a more general type of morphism between anomalous actions (cf.\ \cite[Definition A]{GINE23} and Definition \ref{defn:weakcocyclemorphism}), which we expect to become relevant beyond the present work.

\begin{definition} \label{introdef:Gcoherent}
Let $(A,\alpha,\fu)$ and $(B,\beta,\fv)$ be anomalous actions of a discrete group $G$.
A \emph{$G$-coherent morphism} from $(A,\alpha,\fu)$ to $(B,\beta,\fv)$ is a tuple of linear maps $\phi=(\phi_g : A\rightarrow B)_{g\in G}$ such that 
\begin{enumerate}[leftmargin=*,label=\textup{(\roman*)}]
    \item $\phi_g(a)^*\phi_g(b)=\phi_1(a^*b)$,
    \item $\phi_{gh}(\fu_{g,h}^*\alpha_g(a)b)=\fv_{g,h}^*\beta_g(\phi_h(a))\phi_g(b)$
\end{enumerate}
for all $g,h\in G$ and $a,b\in A$.    
\end{definition}

Motivated by the insights from \cite{ELL10, SZ21}, Neagu and the first-named author \cite{GINE23} study the category of actions of a unitary tensor category $\cC$ on C$^*$-algebras denoted by $C^*_{\cC}$.
This category is shown to admit a natural notion of approximate unitary equivalence and a $\cC$-equivariant Elliott intertwining argument.
An $\omega$-anomalous action of $G$ on $A$, with $3$-cocycle $\omega\in Z^3(G,\bT)$, naturally induces an action of the unitary tensor category $\Hilb(G,\omega)$ on $A$; see e.g.\ \cite[Section 5.2]{EVGI23}.
From this point of view, the notion of a $G$-coherent morphism between $\omega$-anomalous actions coincides with the notion of a morphism in $C^*_{\Hilb(G,\omega)}$. 

The flexibility of the notion of a $G$-coherent morphism is crucial for establishing our existence theorem (Theorem \ref{thm: existence}).
However, this same flexibility makes uniqueness results up to approximate unitary equivalence more subtle.
For this purpose, we introduce a notion of approximate and asymptotic Murray--von Neumann equivalence between $G$-coherent morphisms, inspired by the corresponding notion for $*$-homomorphisms introduced by Gabe \cite{GA20} (Definition \ref{defn:MvN}).
Approximate Murray--von Neumann equivalence provides the appropriate framework for our uniqueness result (Proposition \ref{prop:uniqueness}).
As an analogue of a crucial observation by Gabe, we show that the notion of approximate Murray--von Neumann equivalence turns out to coincide with proper approximate unitary equivalence for $G$-coherent morphisms between separable C$^*$-algebras when the target is either stable or has almost stable rank one; see Proposition \ref{prop:stablemvn}.
We note that our strategy of proof is somewhat different from Gabe's and can thus be viewed as an alternative proof for \cite[Proposition 3.13]{GA20}. 
It entails an extra structural observation about such \cstar-algebras (Lemma \ref{lem:sequencesettingLemma}) that also enables the proof of Theorem \ref{thm: existence}, which we are currently unable to obtain for general separable \cstar-algebras.
However, we note that the structural hypothesis for the C$^*$-algebras in our existence results (Theorem \ref{thm: existence} and Corollary \ref{cor:unital-existence}) cover all \cstar-algebras in the scope of the Elliott classification programme \cite{ZH92, Rordam04, RO16}.

Therefore, as a consequence of the techniques and results we obtain as discussed in the previous paragraph, we may in fact deduce the classification of finite group $G$-kernels with the Rokhlin property on all C$^*$-algebras classified by the Elliott programme, but at the expense of using the extended invariant $\underline{K}T_u$ introduced in \cite{CLASS} that classifies $*$-homomorphisms up to approximate unitary equivalence (see Theorem \ref{thm: classificationKTu}).
To obtain Theorem \ref{introthrm:Gkerclass} we then follow a similar strategy to Izumi in \cite{IZU04II}. We use the complete cohomological triviality of the K-theory modules given by Theorem \ref{thm:generalcohomologyvanishing} to show that in the case of finite group Rokhlin $G$-kernels on real rank zero C$^*$-algebras, the $G$-module structure of $\underline{K}T_u$ is recovered from the K-theory $G$-modules, and hence Theorem \ref{thm: classificationKTu} is the core ingredient that allows us to obtain Theorem \ref{introthrm:Gkerclass}.
We are tempted to conjecture that outside the real rank zero case, the classifying invariant in Theorem \ref{thm: classificationKTu} may also be simplified to the module structure of the classifying invariant KT$_u$, so as to not require the full information in the extended invariant.
However, this must involve a careful analysis of the structure of the $G$-module of Hausdorffised algebraic $K_1$ for a finite group $G$-kernel with the Rokhlin property.
We do not explore the latter in this paper.

In analogy with Jones' subfactor theory \cite{JO83}, and in particular Popa’s classification of amenable finite-index subfactors of $\cR$ \cite{PO94,PO95}, there has been growing interest in the existence, classification and structure of actions of unitary tensor categories on simple amenable C$^*$-algebras \cite{IZ93,CHHPJO24,ARKIKU23,GPKINE24,EVGPJO24,EVJO25,KI24,HPMU26}.
We expect that the techniques underlying our existence and uniqueness results for $G$-coherent morphisms will provide a template for future classification results for anomalous actions and, more generally, for actions of unitary tensor categories.

We briefly discuss the structure of the paper.
In Section \ref{subsec:anomaction} we recall the necessary background on anomalous actions.
In Section \ref{subsec:Gcoher} we recall the $G$-coherence category and show some structural results for morphisms.
In Section \ref{sec:MvNequivalence} we introduce approximate and asymptotic Murray von Neumann equivalence in the $G$-coherence category and its relation to proper approximate and asymptotic unitary equivalence.
In Section \ref{sec:exist+uniq} we prove the existence and uniqueness theorems for anomalous actions with the Rokhlin property.
We then apply these results in Section \ref{sec:classification} to obtain various classification results for anomalous actions on simple amenable \cstar-algebras. In Section \ref{sec:cohomologyvanish}, we treat the K-theoretic obstruction results for Rokhlin $G$-kernels, which combined with the classification results of Section \ref{sec:cohomologyvanish}, allows us to obtain Theorem \ref{introthrm:Gkerclass}. Finally, in Section \ref{sec:range} we obtain the range result Theorem \ref{introthm:rangeofinvstable} .

\subsection*{Acknowledgements} 
SGP and GS were both supported by research project G085020N funded by the Research Foundation Flanders (FWO).
SGP was also supported by postdoctoral fellowship 1249225N funded by the FWO.
GS was furthermore supported by the European Research Council under the European Union's Horizon Europe research and innovation programme (ERC grant AMEN--101124789).

For the purpose of open access, the authors apply a CC BY public copyright license to any author accepted manuscript version arising from the submission of this work.

Both authors thank an anonymous referee for providing valuable feedback on an earlier version of this article, which substantially improved the exposition in the current version.

\subsection*{Tool and computational resource disclosure}
Generative AI, in the form of Large Language Models or otherwise, has not been used by the authors in any capacity to discover or prove any of the results in this manuscript.
We used ChatGPT solely for the purpose of English language editing of specific parts during the final stages of the writing before first submission to arXiv.

\section{Preliminaries}\label{sec:prelims}
\renewcommand{\thetheorem}{\thesection.\arabic{theorem}}
\setcounter{theorem}{0}

\begin{notation}
Throughout this paper $A$ and $B$ denote C$^*$-algebras and $G$ will be a countable discrete group that is often assumed to be finite.
We denote the positive elements of $A$ by $A^+$ and the unit ball of $A$ by $A_{1}$.
We denote the multiplier algebra of $A$ by $\cM(A)$ and its unitary group by $\cU\cM(A)$.
We denote by $\cU(\eins+A)$ the intersection $(\eins+A)\cap \cU( A^{\dagger})$ where $A^{\dagger}$ denotes the (forced) unitisation of $A$ and $\bf{1}$ is the unit of $A^{\dagger}$.
A sequential approximate unit $e_n$ of $A$ is called quasicentral with respect to a subalgebra $B\subset \cM(A)$ if
\[
\lim_{n\rightarrow \infty} \|e_na-ae_n\|=0,\quad a\in B.
\]
Moreover, $e_n$ is called idempotent if $e_{n+1}e_{n}=e_n$ for all $n\in \bN$.
For a $\sigma$-unital C$^*$-algebra $A$ and a separable C$^*$-subalgebra $B\subset \cM(A)$ there always exists an idempotent approximate unit of $A$ which is quasicentral with respect to $B$.
If $a,b\in A$ are such that $\|a-b\|<\varepsilon$ for some $\varepsilon>0$ we will often write $a=_{\varepsilon}b$.

We denote by $A_{\infty}=l^{\infty}(\mathbb{N},A)/c_0(\mathbb{N},A)$ and by $A_{\infty}\cap A'$ the subalgebra of $A_{\infty}$ consisting of central sequences.
We denote the annihilator subalgebra of $A$ in $A_{\infty}$ by $A_{\infty}\cap A^{\perp}=\{ a\in A_{\infty}: ab=ba=0\ \forall b\in A\}$ and Kirchberg's central sequence algebra by $F_\infty(A)=(A_{\infty}\cap A')/(A_{\infty}\cap A^{\perp})$, which is a unital C$^*$-algebra whenever $A$ is $\sigma$-unital.
Any automorphism $\theta\in \Aut(A)$ induces an automorphism of $A_{\infty}\cap A'$ that preserves the subalgebra $A_{\infty}\cap A^{\perp}$.
Therefore $\theta$ induces an automorphism of $F_\infty(A)$ which we also denote by $\theta$ with slight abuse of notation.
Denoting by $\Inn(A)$ the normal subgroup of $\Aut(A)$ consisting of inner automorphisms of the form $\Ad(u)$ for $u\in\cU\cM(A)$, it is easy to see that every inner automorphism induces the trivial map on $F_\infty(A)$.
This construction hence yields a well-defined group homomorphism 
\[
\Out(A):=\Aut(A)/\Inn(A)\ni [\theta] \mapsto \theta\in\Aut(F_\infty(A)).
\]
Let $\Gamma$ be an abelian group and $n\in \bN$. We denote  
\[
n\Gamma=\{ng: g\in \Gamma\},\quad {}_n \Gamma=\{g\in \Gamma: ng=0\},\quad\text{and}\quad \Gamma_n=\Gamma/n\Gamma.
\]
Throughout the article, we shall give various lengthy arguments involving maps that arise as compositions of several other maps.
To lighten notation, we shall mostly omit the composition symbol $\circ$, for instance we would write $\varphi\alpha\varphi^{-1}$ in place of $\varphi\circ\alpha\circ\varphi^{-1}$.
\end{notation}

\subsection{Anomalous actions and the Rokhlin property}\label{subsec:anomaction}

In this section we recall the notion of an anomalous action of a group $G$ and the Rokhlin property for anomalous actions when $G$ is finite.
We refer the reader to \cite{THESIS,GP25} for further details.

\begin{definition}
    An \emph{anomalous action} of a discrete group $G$ on a C$^*$-algebra $A$ is a pair $(\alpha,\fu)$ consisting of maps $\alpha:G\rightarrow \Aut(A)$ and $\fu:G\times G\rightarrow \cU\cM(A)$ such that the following holds for all $g,h,k\in G$:
    \begin{enumerate}[leftmargin=*]
        \item $\alpha_g\alpha_h=\Ad(\fu_{g,h})\alpha_{gh}$,\label{item:multiplicativity}
        \item $\alpha_g(\fu_{h,k})\fu_{g,hk}\fu_{gh,k}^*\fu_{g,h}^*\in \bT \eins_{\cM(A)}$,\label{item:anomaly}
        \item $\fu_{g,1}=\fu_{1,g}=\eins$.
    \end{enumerate}
It follows from \eqref{item:multiplicativity} that the function  $\lambda_{g,h,k}=\alpha_g(\fu_{h,k})\fu_{g,hk}\fu_{gh,k}^*\fu_{g,h}^*$ yields a normalised $3$-cocycle with coefficients in the circle group. 
We denote it by $o(\alpha,\fu)\in Z^3(G,\bT)$ and call it the \emph{anomaly} of $(\alpha,\fu)$. 
We will often refer to the triple $(A,\alpha,\fu)$ as an anomalous action of the group $G$ as well.
\end{definition}

\begin{remark}
Anomalous actions are closely related to $G$-kernels.
A \emph{$G$-kernel} is a homomorphism $\theta:G\rightarrow \Out(A)$.
Following the procedure of \cite[Section 3.2]{THESIS} one can pick a lift $\widetilde{\theta}:G\rightarrow \Aut(A)$ and unitaries $\fu_{g,h}\in \cU\cM(A)$ with $\fu_{g,1}=\fu_{1,g}=\eins$ such that $\pi(\widetilde{\theta}_g)=\theta_g$ and $\widetilde{\theta}_g\widetilde{\theta}_h=\Ad(\fu_{g,h})\widetilde{\theta}_{gh}$ where $\pi:\Aut(A)\rightarrow \Out(A)$ is the quotient map.
If $Z\cM(A)=\bC$, for example if $A$ is simple, then $\widetilde{\theta}_g(\fu_{h,k})\fu_{g,hk}\fu_{gh,k}^*\fu_{g,h}^*\in \bT$. 
Thus $(\widetilde{\theta},\fu)$ is an anomalous action and the cohomology class of $o(\widetilde{\theta},\fu)$ in $H^3(G,\bT)$ coincides with the lifting obstruction of $\theta$.
For a discussion on the differences between anomalous actions, $G$-kernels and their associated invariants see for example \cite[Section 2.3]{JO21}.

Having two $G$-kernels $\alpha: G\to\Out(A)$ and $\beta: G\to\Out(B)$ with respective lifts $\widetilde{\alpha}: G\to\Aut(A)$ and $\widetilde{\beta}: G\to\Aut(B)$, we recall that $\alpha$ and $\beta$ are said to be \emph{conjugate} if there exists an isomorphism $\varphi: A\to B$ such that $\alpha_g=\pi(\varphi^{-1}\widetilde{\beta}_g\varphi)$ for all $g\in G$.
\end{remark}

We will be interested in anomalous actions with the Rokhlin property.

\begin{definition} \label{def:Rp}
Let $G$ be a finite group.
A $G$-kernel $\alpha$ on a separable C$^*$-algebra $A$ is said to have the \emph{Rokhlin property} if there exists a projection $p\in F_\infty(A)$ such that $\sum_{g\in G}\alpha_g(p)=\eins$.
We call a projection $p$ as above a \emph{Rokhlin projection}.

An anomalous action $(A,\alpha,\fu)$ of a finite group $G$ is said to have the Rokhlin property if its associated $G$-kernel has the Rokhlin property in the above sense.
\end{definition}

For $\omega\in Z^3(G,\bT)$ the $\omega$-left regular representation on $C(G)$ given by 
$(\alpha,\fu)$ with 
\[
\alpha_g(f)(h)=f(g^{-1}h) \quad\text{and}\quad \fu_{g,h}(k)=\omega(k^{-1},g,h)
\]
is an $\omega$-anomalous action on $C(G)$ such that the delta function at $1_G$ is a Rokhlin projection for $(\alpha,\fu)$.
From the $\omega$-left regular action one can also induce a Rokhlin $\omega$-anomalous action on the UHF algebra $M_{|G|^\infty}$ (see \cite[Proposition 6.1]{GP25}).
For more examples of anomalous actions with the Rokhlin property see \cite[Section 6]{THESIS}. 

In applications, we shall frequently use the Rokhlin property in the following form, the proof of which is a standard unpacking of definitions:

\begin{lemma} \label{lem:Rp}
Let $G$ be a finite group and $\alpha$ a $G$-kernel on a separable \cstar-algebra $A$ with lifting map $\widetilde{\alpha}: G\to\Aut(A)$.
Then $\alpha$ has the Rokhlin property if and only if there exists a sequence of positive contractions $f_n\in A$ satisfying the following properties for all $a\in A$ and $g,h\in G$:
	\begin{enumerate}[leftmargin=*,label=\textup{(\roman*)}]
	\item $\| [\widetilde{\alpha}_g(f_n),a]\| \to 0$, 
    \item $\| \big( \widetilde{\alpha}_g(f_n)\widetilde{\alpha}_h(f_{n})-\delta_{g,h}\widetilde{\alpha}_g(f_{n}) \big) a\|\to 0$, 
    \item $\|\big(\eins-\sum_{k\in G}\widetilde{\alpha}_k(f_{n}) \big) a\|\to 0$.
	\end{enumerate} 
\end{lemma}

\begin{notation}
Let $\alpha:G\rightarrow \Out(A)$ and $\beta:G\rightarrow \Out(B)$ be two $G$-kernels with scalar valued lifting obstructions $\ob(\alpha),\ob(\beta)\in H^3(G,\bT)$.
We may consider the tensor product $G$-kernel $\alpha\otimes \beta:G\rightarrow \Out(A\otimes B)$ by declaring that $(\alpha\otimes\beta)_g$ is represented by $\widetilde{\alpha}_g\otimes\widetilde{\beta}_g$, if $\widetilde{\alpha}: G\to\Aut(A)$ and $\widetilde{\beta}: G\to\Aut(B)$ are lifting maps for $\alpha$ and $\beta$, respectively.\footnote{This works both for the reduced and full tensor products.}
The lifting obstruction of $\alpha\otimes\beta$ is then equal to the product $\ob(\alpha)\ob(\beta)$.
\end{notation}

\begin{lemma}\label{lem:Rokhlintensor}
Let $G$ be a finite group and $\alpha:G\rightarrow \Out(A)$ and $\beta:G\rightarrow \Out(B)$ two $G$-kernels on separable C$^*$-algebras.
If either $\alpha$ or $\beta$ has the Rokhlin property, then so does $\alpha\otimes \beta$.
\end{lemma}
\begin{proof}
We may assume that $\alpha$ has the Rokhlin property.
Let $e_n$ be an approximate unit for $B$.
Then $\theta(e_n)$ is also an approximate unit for $B$ for any $\theta\in\Aut(B)$.
Therefore $\beta_g(e_n)-e_n\in B_\infty\cap B^{\perp}$ for all $g\in G$.
This yields a unital $G$-equivariant $*$-homomorphism
\begin{align*}
\iota:(F_\infty(A),\alpha)&\rightarrow (F_\infty(A\otimes B),\alpha\otimes\beta)\\
[(a_n)_n]&\mapsto [(a_n\otimes e_n)_n].
\end{align*}
If $p\in F_\infty(A)$ is a Rokhlin projection for $\alpha$ then $\iota(p)$ is clearly also a Rokhlin projection for $\alpha\otimes \beta$.
This implies the claim.
\end{proof}

\subsection{$G$-coherent morphisms}\label{subsec:Gcoher}

We now set up the categorical rudiments necessary for the classification of anomalous actions.
Our notion of morphism between anomalous actions is borrowed from the notion of an equivariant morphism between their induced actions of the tensor category of the group $G$ twisted by a $3$-cocycle.
In particular our definition should be considered a special case of \cite[Definition A]{GINE23}.

\begin{definition} \label{defn:weakcocyclemorphism}
Let $(A,\alpha,\fu)$ and $(B,\beta,\fv)$ be anomalous actions of a group $G$.
A \emph{$G$-coherent morphism} from $(A,\alpha,\fu)$ to $(B,\beta,\fv)$ is a tuple of linear maps $\phi=(\phi_g:A\rightarrow B)_{g\in G}$ such that for all $g,h\in G$ and $a,b\in A$, one has
\begin{enumerate}[label=\textup{(\roman*)},leftmargin=*]
    \item $\phi_g(a)^*\phi_g(b)=\phi_1(a^*b)$, \label{item:1weakcocyc}
    \item $\phi_{gh}(\fu_{g,h}^*\alpha_g(a)b)=\fv_{g,h}^*\beta_g(\phi_h(a))\phi_g(b)$. \label{item:2weakcocyc}
\end{enumerate}

\end{definition}

\begin{remark}
It follows from condition \ref{item:1weakcocyc} of Definition \ref{defn:weakcocyclemorphism} (upon inserting $g=1_G$) that for any $G$-coherent morphism $\phi:(A,\alpha,\fu)\rightarrow (B,\beta,\fv)$ between anomalous actions, the map $\phi_1$ is a $*$-homomorphism. Indeed, a straightforward computation using condition \ref{item:1weakcocyc} shows the $2$-positivity of $\phi_1$ and as any $a\in A$ is in the multiplicative domain of $\phi_1$, the claim follows.
\end{remark}

\begin{definition}{(cf.\ \cite[Definition 1.10]{SZ21})} \label{def:cocyclemorphism}
Let $(A,\alpha,\fu)$ and $(B,\beta,\fv)$ be anomalous actions of a group $G$.
A \emph{cocycle morphism} consists of an extendible $*$-homomorphism $\phi: A\rightarrow B$ and a map $\bv: G\rightarrow \cU\cM(B)$ satisfying
    \begin{enumerate}[label=\textup{(\roman*)},leftmargin=*]
        \item $\Ad(\bv_g)\beta_g\phi=\phi\alpha_g$, \label{def:cocyclemorphism:1}
        \item $\bv_g\beta_g(\bv_h)\fv_{g,h}\bv_{gh}^*=\phi^{+}(\fu_{g,h})$ \label{def:cocyclemorphism:2}
    \end{enumerate}
for all $g,h\in G$.
If $\phi$ is invertible, we call $(\phi,\bv)$ a \emph{cocycle conjugacy}.\footnote{We refer to \cite{SZ21} for the details of the map $\phi^{+}$ when $\phi$ is extendible. We will mostly care about the case that $\phi$ is non-degenerate, in which case $\phi^+$ is simply the unique extension of $\phi$ to a unital $*$-homomorphism $\cM(A)\rightarrow \cM(B)$.}
\end{definition}

\begin{remark} \label{rem:coc-morphs-induce-G-coh}
Cocycle morphisms $(\phi,\bv)$ always induce $G$-coherent morphisms by considering the family of maps $\phi_g(\cdot)=\bv_g^*\phi(\cdot)$ for $g\in G$, as can be checked from the defining conditions easily.
In particular, if $\varphi: A\to B$ is genuinely equivariant (i.e., it satisfies the properties above in place of $\phi$ with $\bv=\eins$), then the tuple $(\varphi)_{g\in G}$ is also a $G$-coherent morphism in a trivial way.
The passage from cocycle morphisms to $G$-coherent morphisms is generally not a reversible procedure, however, which is why the notion of $G$-coherent morphisms comes with the advantage of more flexibility that cocycle morphisms do not always have.
To see that the above assignment is not necessarily surjective, consider any non-unital \cstar-algebra $A$ and a unitary representation $\bu: G\to\cU\cM(A)$ whose range does not lie in $\cU(A^\dagger)$.
Then the family of maps $\phi_g: A\to A^\dagger$ given by $\phi_g(a)=\bu_g^*a$ is easily seen to form a $G$-coherent morphism when we equip $A^\dagger$ with the trivial $G$-action and $A$ with the action $\alpha_g=\Ad(\bu_g)$.
If one were to assume that this was induced from a cocycle morphism as above, it would immediately force $\bu_g\in\cU(A^\dagger)$ for $g\in G$, contradicting the starting assumption.
\end{remark}

We nevertheless observe the following property of $G$-coherent morphisms, which in a sense asserts that they bear a great deal of conceptual similarity with cocycle morphisms.

\begin{proposition} \label{prop:leftinnerproduct}
Let $\phi:(A,\alpha,\fu)\rightarrow (B,\beta,\fv)$ be a $G$-coherent morphism between anomalous actions.
Then there exists a unique tuple of partial isometries $(\bv_g)_{g\in G}$ in $B^{**}$ such that the following holds for all $g,h\in G$ and $a\in A$:
	\begin{enumerate}[label=\textup{(\alph*)},leftmargin=*]
    \item $\bv_1$ is the support projection of $\phi_1(A)\subseteq B$,\label{prop:leftinnerproduct:a}
    \item $\phi_g(a)=\bv_g^*\phi_1(a)$, \label{prop:leftinnerproduct:b}
    \item $\bv_g^*\bv_g=\bv_1$ and $\bv_g\bv_g^*=\beta_g(\bv_1)$,\footnote{With slight abuse of notation, we write $\beta_g$ in place of $\beta_g^{**}$ for the induced automorphism on the double dual.} \label{prop:leftinnerproduct:c}
    \item $\bv_g\beta_g(\bv_h)\fv_{g,h}\bv_{gh}^*=\phi_1^{**}(\fu_{g,h})$. \label{prop:leftinnerproduct:d}
    \end{enumerate}
In particular $\phi$ also satisfies
	\begin{equation}\label{eqn:leftinnerproduct}
   \phi_g(\alpha_g(a))\phi_g(\alpha_g(b))^*=\beta_g(\phi_1(ab^*)),\quad g\in G,\ a,b\in A.
	\end{equation}
\end{proposition}
\begin{proof}
By functoriality, both of the anomalous actions extend to anomalous actions on the double duals $A^{**}$ and $B^{**}$, respectively. 
By \cref{defn:weakcocyclemorphism}\ref{item:1weakcocyc} each map $\phi_g$ is clearly bounded and linear, so it extends to a weak$^*$-continuous map $\phi_g^{**}:A^{**}\rightarrow B^{**}$ between the double duals.
As the involution is weak$^*$-continuous and multiplication is separately weak$^*$-continuous, it follows that the tuple of maps $(\phi_g^{**})_{g\in G}$ still satisfies the conditions in Definition \ref{defn:weakcocyclemorphism}, i.e., it can be understood as a $G$-coherent morphism between the double duals.

Set $\bv_g=\phi_g^{**}(\eins)^*$ for each $g\in G$. 
Let $(e_\lambda)_{\lambda\in\Lambda}$ be an increasing approximate unit for $A$, which thereby converges to $\eins_{A^{**}}$ in the weak$^*$ topology.
Then clearly $\bv_1=\lim_\Lambda \phi_1(e_\lambda)$ is the support projection of the range of $\phi_1$.
This verifies \ref{prop:leftinnerproduct:a}. 
As a consequence of Definition \ref{defn:weakcocyclemorphism}\ref{item:2weakcocyc} with $g=1$ and $a=\eins$, one has the equation
\[
\phi_h^{**}(b)=\phi_h^{**}(\eins)\phi_1^{**}(b)=\bv_h^*\phi_1^{**}(b),\quad b\in A^{**}.
\]  
Upon considering $b\in A$, this yields \ref{prop:leftinnerproduct:b}.
Using Definition \ref{defn:weakcocyclemorphism}\ref{item:1weakcocyc}, one has that
\[
\bv_g\bv_g^* = \phi_g^{**}(\eins)^*\phi_g^{**}(\eins)=\phi_1^{**}(\eins) = \bv_1,\quad g\in G.
\] 
Applying Definition \ref{defn:weakcocyclemorphism}\ref{item:2weakcocyc} with $a=b=\eins_{A^{**}}$ yields
\[
\bv_{gh}^* \phi_1^{**}(\fu_{g,h}^*) = \fv_{g,h}^* \beta_g(\bv_h^*) \bv_g^*.
\]
After considering the adjoint on both sides, we may multiply from the right with $\bv_{gh}^*$ and use the above equation to get
\[
\bv_g\beta_g(\bv_h)\fv_{g,h}\bv_{gh}^* = \phi_1^{**}(\fu_{g,h}) \bv_{gh}\bv_{gh}^* = \phi_1^{**}(\fu_{g,h}).
\]
This yields \ref{prop:leftinnerproduct:d}.
We shall now verify the second half of \ref{prop:leftinnerproduct:c}.
Applying Definition \ref{defn:weakcocyclemorphism}\ref{item:2weakcocyc} with $h=1$, one has
\[
\bv_g^* = \phi_g^{**}(\eins) = \beta_g(\phi_1^{**}(\eins))\phi_g^{**}(\eins) = \beta_g(\bv_1)\bv_g^*.
\]
Thus $\bv_g^*\bv_g\leq \beta_g(\bv_1)$.
On the other hand, applying \ref{prop:leftinnerproduct:d} to the pair $(g^{-1},g)$ yields
\[
\phi_1^{**}(\fu_{g^{-1},g})=\bv_{g^{-1}} \beta_{g^{-1}}(\bv_g)\fv_{g^{-1},g}=\bv_{g^{-1}}\fv_{g^{-1},g}\beta_g^{-1}(\bv_g).
\]
Multiplying on the left by $\phi_1^{**}(\fu_{g^{-1},g})^*$ and applying $\beta_g$ one has that
\[
\beta_g(\bv_1)=\beta_g\big( \phi_1^{**}(\fu_{g^{-1},g})^* \fv_{g^{-1},g} \bv_{g^{-1}}  \big) \bv_g.
\]
Thus $\beta_g(\bv_1)\leq\bv_g^*\bv_g$ and we get equality by the previous step.
For the ``in particular'' part, we may use Definition \ref{defn:weakcocyclemorphism}\ref{item:2weakcocyc} with $h=1$ to see that
\begin{align*}
    \phi_g(\alpha_g(a))\phi_g(\alpha_g(b))^* &= \lim_\Lambda \phi_g(\alpha_g(a)e_\lambda)\phi_g(\alpha_g(b)e_\lambda)^*\\
    &=\lim_{\Lambda} \beta_g(\phi_1(a))\phi_g(e_\lambda)\phi_g(e_\lambda)^*\beta_g(\phi_1(b))^* \\
    &=\beta_g(\phi_1(a))\bv_g^*\bv_g\beta_g(\phi_1(b))^*\\
    &=\beta_g(\phi_1(a))\beta_g(\bv_1)\beta_g(\phi_1(b))^*\\
    &=\beta_g(\phi_1(ab^*))
\end{align*}
for all $a,b\in A$ and $g\in G$.
\end{proof}

We may compose $G$-coherent morphisms by composing the linear maps with matching indices.
This yields a category whose objects are anomalous $G$-actions and whose morphisms are $G$-coherent morphisms between them.
The identity morphism of an anomalous $G$-action is simply given by the trivial tuple consisting of the identity map, indexed over $G$.
Our next observation entails that isomorphisms in this category coincide with cocycle conjugacies.
For this reason, we shall not distinguish between cocycle conjugacies and $G$-coherent isomorphisms for the rest of the article.

\begin{lemma}\label{lem: cocycleisoscoincide}
Let $\phi: (A,\alpha,\fu)\rightarrow (B,\beta,\fv)$ be a $G$-coherent morphism such that $\phi_1$ is a non-degenerate $*$-homomorphism.
Then the tuple of partial isometries $(\bv_g)_{g\in G}$ in $B^{**}$ from Proposition~\ref{prop:leftinnerproduct} is in fact a tuple in $\cU\cM(B)$.
In particular, $\phi$ is uniquely induced from a cocycle morphism as in Remark~\ref {rem:coc-morphs-induce-G-coh}.
Furthermore, two anomalous $G$-actions are $G$-coherently isomorphic if and only if they are cocycle conjugate.
\end{lemma}
\begin{proof}
Let $(\bv_g)_{g\in G}$ be the tuple of partial isometries from Proposition~\ref{prop:leftinnerproduct}.
Since $\bv_1$ is the support projection of $\phi_1(A)$, which by assumption contains an approximate unit for $B$, it follows immediately that $\bv_1=\eins$, and that $\bv_g$ is a unitary for each $g\in G$.
Using the nondegeneracy of $\phi_1$, we can see that
\[
\bv_g^*\cdot B = \bv_g^*\cdot\overline{\phi_1(A)B} = \overline{\phi_g(A)B}\subseteq B.
\]
Let $(e_\lambda)$ be an increasing approximate unit for $A$.
Clearly $\beta_g\phi_1$ is also a nondegenerate $*$-homomorphism, hence we can see from condition \eqref{eqn:leftinnerproduct} that
\[
b=\lim_\Lambda \beta_g(\phi_1(e_\lambda))b = \lim_\Lambda \phi_g(\alpha_g(e_\lambda^{1/2}))\phi_g(\alpha_g(e_\lambda^{1/2}))^* b \in \overline{\phi_g(A)B}
\]
for all $b\in B$.
This implies $\bv_g^*\cdot B = B$.
Since $\bv_g$ is a unitary, we also obtain $\bv_g\cdot B=B$ and hence $\bv_g\in\cU\cM(B)$.\footnote{Recall that multipliers can be defined as those elements $x\in B^{**}$ with $xB+Bx\subseteq B$.}

Hence $\phi_g^{**}$ restricts to a linear map $\cM(A)\to\cM(B)$ for each $g\in G$, and condition \ref{prop:leftinnerproduct:d} yields that they satisfy the cocycle identity in Definition \ref{def:cocyclemorphism}\ref{def:cocyclemorphism:2}.
Furthermore, it follows for $a\in A^+$ and $g\in G$ that
    \[
    \Ad(\bv_g^*)\phi_1(\alpha_g(a))=\phi_g(\alpha_g(a^{1/2}))\phi_g(\alpha_g(a^{1/2}))^*\stackrel{\eqref{eqn:leftinnerproduct}}{=} \beta_g(\phi_1(a)).
    \]
By conjugating with $\bv_g$ on both sides and keeping in mind that all the involved maps are linear, we obtain $\phi_1\alpha_g = \Ad(\bv_g)\beta_g\phi_1$ for all $g\in G$.
So we see that $(\phi_1,\bv)$ is a cocycle morphism that uniquely induces $\phi$ as per Remark~\ref {rem:coc-morphs-induce-G-coh}.
It is clearly a cocycle conjugacy if we assume that $\phi_1$ is an isomorphism, which is the case if $\phi$ was assumed to be a $G$-coherent isomorphism.
This finishes the proof.
\end{proof}
   
For an anomalous action $(A,\alpha,\fu)$, any unitary $u\in \cU\cM(A)$ induces a cocycle morphism $\Ad(u):(A,\alpha,\fu)\rightarrow (A,\alpha,\fu)$, which in the language of $G$-coherent morphisms is given by
\[
\Ad(u)_g(a)=\alpha_g(u)au^*,\quad g \in G,\ a\in A.
\]
As discussed for example in \cite{EVGI23}, anomalous actions of a group $G$ with anomaly $\omega$ induce actions of the tensor category  $\Hilb(G,\omega)$.
Then $G$-coherent morphisms between anomalous actions coincide with the notion of equivariant morphisms between the induced actions of $\Hilb(G,\omega)$.
In particular, it follows from \cite[Section 4]{GINE23} that the category whose objects are $\omega$-anomalous actions and morphisms are $G$-coherent morphisms admits inductive limits (see also the discussion in \cite[Section 6.1]{THESIS}).
These observations allow us to collect a few notions and results from the intertwining machinery developed in \cite{GINE23} when restricted to the case of $\Hilb(G,\omega)$ actions.

\begin{definition}[cf.\ {\cite[Definition C]{GINE23}}]
Let $(A,\alpha,\fu)$ and $(B,\beta,\fv)$ be anomalous actions of a countable group $G$.
Assume that $A$ is separable.
Let $\phi,\psi:(A,\alpha,\fu)\rightarrow (B,\beta,\fv)$ be two $G$-coherent morphisms.
We say that $\phi$ and $\psi$ are \emph{proper approximate unitarily equivalent} if there exists a sequence of unitaries $u_n\in \cU(\eins+B)$ such that
\[
\lim_{n\rightarrow \infty} \beta_g(u_n)\phi_g(a)u_n^*=\psi_g(a)
\]
for all $g\in G$ and $a\in A$. 
We denote this relation by $\phi\approx_{\mathrm{pu}}^G \psi$. 
\end{definition}

When $G$ is the trivial group and $A$ and $B$ are equipped with the trivial action, then proper approximate unitary equivalence of the $*$-homomorphisms $\phi,\psi:A\rightarrow B$ is just denoted $\phi\approx_{\mathrm{pu}}\psi$.
When $B$ is unital, then proper approximate unitary equivalence coincides with the notion of approximate unitary equivalence where one takes $u_n\in \cU(B)$, which is simply denoted by $\phi\approx_{\mathrm{u}}^G \psi$ and if $G$ is trivial by $\phi\approx_{\mathrm{u}}\psi$.

As a consequence of \cite[Theorem D, Remark 6.6]{GINE23}, we get the following generalisation of the classical two-sided Elliott intertwining argument \cite[Corollary 2.3.4]{RO02} (see also \cite{ELL10}).
This rests on the basic observation that the set of $G$-coherent morphisms between two anomalous $G$-actions carries a natural topology, whereby the tuples $(\phi_g)_g$ and $(\psi_g)_g$ are close if for a large finite set $F\subset G$, the maps $\phi_g$ and $\psi_g$ are close in the point-norm topology for every $g\in F$.
When $G$ is countable and the involved \cstar-algebras are separable, it is a routine argument that this defines a Polish topology induced by a metric satisfying the abstract requirements of \cite{ELL10}.
This can be considered a superior feature of $G$-coherent morphisms as opposed to another, much less well-behaved, topology considered in \cite{SZ21} on the set of cocycle morphisms between two separable \cstar-dynamical systems.

We would like to point out a conceptual connection between the intertwining theorem below, generalizing the one in \cite{SZ21} for genuine $G$-actions, and an ad-hoc generalization of that intertwining theorem exhibited by Nawata in the proof of \cite[Theorem 8.1]{Nawata23}.

\begin{theorem}\label{thm: Elliott intertwining}
Let $(A,\alpha,\fu)$ and $(B,\beta,\fv)$ be two anomalous actions of a countable discrete group $G$ with $o(\alpha,\fu)=o(\beta,\fv)$ on separable C$^*$-algebras.
Let $\phi:(A,\alpha,\fu)\rightarrow (B,\beta,\fv)$ and $\psi:(B,\beta,\fv)\rightarrow (A,\alpha,\fu)$ two $G$-coherent morphisms satisfying 
\[
\phi\circ\psi\approx_{\mathrm{pu}}^G\id_B \quad\text{and}\quad \psi\circ\phi\approx_{\mathrm{pu}}^G \id_A.
\]
Then $\phi$ and $\psi$ are proper approximately unitary equivalent to mutually inverse cocyle conjugacies.
\end{theorem}

\subsection{Classification of morphisms and total K-theory}\label{sec:totalK}

We now recall some results on the classification of $*$-homomorphisms between C$^*$-algebras up to approximate unitary equivalence, as we will require them in later sections.
We refer to \cite{CLASS} and \cite{IZU04II} for more detailed accounts.
Given $n\in \bN$, the $K$-groups of $A$ with coefficients in $n$ are defined by
\[
K_i(A;\bZ_n):=K_i(A\otimes \mathcal{O}_{n+1}),\ i=0,1.
\]
These give rise to functors.\footnote{There are many equivalent ways to define the K-theory groups with coefficients in $n$ see \cite[Appendix A]{CLASS}.}
There are multiple relevant natural transformations between K-theoretic functors. Firstly, using the notation of Section \ref{sec:prelims}, there is the Künneth short exact sequence
\begin{equation}\label{eqn:Kunneth}
    0\rightarrow K_i(A)_n\xrightarrow{\rho_{A,n}^{i}}K_i(A;\bZ_n)\xrightarrow{\beta_{A,n}^i} {}_n K_{1-i}(A)\rightarrow 0.
\end{equation}
There are also the \emph{Bockstein maps}
\begin{equation}\label{eqn:Bocksteinmaps}
\kappa_{A,n,m}^i: K_i(A;\bZ_m)\rightarrow K_i(A;\bZ_n).
\end{equation}
These natural transformations are compatible with one another, in particular one has a commuting diagram
\begin{equation}\label{diag:compatibility}
\begin{tikzcd}
    K_i(A)_n\ar[r,"\rho^i_{A,n}"]\ar[d,"\times \frac{m}{(m,n)}"]& K_i(A,\bZ_n)\ar[r,"\beta^i_{A,n}"]\ar[d,"\kappa^i_{A,m,n}"]& {}_n K_{1-i}(A)\ar[d,"\times \frac{n}{(m,n)}"]\\
    K_i(A)_m\ar[r,"\rho^i_{A,m}"]& K_i(A;\bZ_m)\ar[r,"\beta^i_{A,m}"]& {}_m K_{1-i}(A) 
\end{tikzcd}
\end{equation}
where we denote $\gcd(m,n)$ by $(m,n)$ for brevity.

\begin{definition}
The \emph{total K-theory} of a C$^*$-algebra $A$, denoted $\underline{K}(A)$ consists of the collection of abelian groups $K_i(A)$ and $K_i(A;\bZ_n)$ for $n\geq 2$.
A \emph{$\Lambda$-morphism} $\underline{\alpha}:\underline{K}(A)\rightarrow \underline{K}(B)$ consists of a family of group homomorphisms
    \[
    \alpha^i:K_i(A)\rightarrow K_i(B)\quad \text{and}\quad \alpha_n^i:K_i(A;\bZ_n)\rightarrow K_i(B;\bZ_n)
    \]
that intertwine the natural morphisms of \eqref{eqn:Kunneth} and the Bockstein maps of \eqref{eqn:Bocksteinmaps}.
We may compose $\Lambda$-morphisms componentwise.
We denote the set of all $\Lambda$-morphisms from $\underline{K}(A)$ to $\underline{K}(B)$ by $\Hom_{\Lambda}(\underline{K}(A),\underline{K}(B))$.
\end{definition}

For any $*$-homomorphism $\varphi:A\rightarrow B$ we denote by $\underline{K}(\varphi)$ the $\Lambda$-morphism induced by $\varphi$.
Then $\underline{K}(\varphi)$ preserves the order structure, i.e., one has $K_0(\varphi)(K_0(A)_+)\subseteq K_0(B)_+$.
When both $A$, $B$ and $\varphi$ are unital, then $\underline{K}(\varphi)$ also preserves the class of the unit, i.e., one has $K_0(\varphi)([\eins_A])=[\eins_B]$.
We recall the classification of morphisms between Kirchberg algebras in the UCT class by total K-theory.
The result follows from combining \cite{DALO96} and \cite[Example 8.4.14]{RO02} (see also \cite[Theorem 8.10]{GA24}).

\begin{theorem} \label{thm:classificationKirchberg}
Let $A$ and $B$ be Kirchberg algebras in the UCT class. Then
    \begin{enumerate}[leftmargin=*,label=\textup{(\arabic*)}]
    \item Suppose $A$ and $B$ are stable.
    Then for any $\Lambda$-morphism $\underline{\alpha}:\underline{K}(A)\rightarrow \underline{K}(B)$ there exists a $*$-homomorphism $\varphi:A\rightarrow B$ such that $\underline{K}(\varphi)=\underline{\alpha}$. Moreover, $\varphi$ is unique up to proper approximate unitary equivalence.\label{item:classKirchbergstable}
    \item Suppose $A$ and $B$ are unital.
    Then for any $\Lambda$-morphism $\underline{\alpha}:\underline{K}(A)\rightarrow \underline{K}(B)$ such that $\alpha^0([1_A])=[1_B]$ there exists a unital $*$-homomorphism $\varphi:A\rightarrow B$ such that $\underline{K}(\varphi)=\underline{\alpha}$. Moreover, $\varphi$ is unique up to approximate unitary equivalence.\label{item:classKirchbergunital}
    \end{enumerate}
\end{theorem}

Unital, simple, separable, nuclear, tracially AF C$^*$-algebras (TAF-algebras) in the UCT class were classified by Lin in \cite{LIN01,LIN04}.
This class of C$^*$-algebras also admits an analogous classification of $*$-homomorphisms.

\begin{theorem}[cf.\ {\cite[Theorem 1.1]{DA04}}]\label{thm:TAF}
Let $A$ and $B$ be unital, simple, separable, nuclear TAF-algebras in the UCT class.
Then for any $\Lambda$-morphism $\underline{\alpha}:\underline{K}(A)\rightarrow \underline{K}(B)$ such that $\alpha^0([\eins_A])=[\eins_B]$ and $\alpha^0(K_0(A)_+)\subseteq K_0(B)_+$ there exists a unital $*$-homomorphism $\varphi:A\rightarrow B$ such that $\underline{K}(\varphi)=\underline{\alpha}$. 
Moreover, $\varphi$ is unique up to approximate unitary equivalence.
\end{theorem}
\begin{remark}\label{rmk:splitting}
By \cite{BO79,BO80} (see also \cite[Appendix A.3]{CLASS}) the Künneth short exact sequences \eqref{eqn:Kunneth} admit splittings that are compatible with the Bockstein operations.
Thus $\underline{K}(A)$ can be recovered from $K_i(A)$ for $i=0,1$.
However, these splittings may not be natural, so $\underline{K}(A)$ usually carries more information than $K_i(A)$ for $i=0,1$ at the level of morphisms. 
\end{remark}

As a consequence of \cite[Proposition 2.10]{DALO96} the functor $\underline{K}(A)$ is completely recovered in nice cases from the groups $K_i(A;\bZ_{p^m})$ and $K_i(A)$ for primes $p$, $m\in \bN$ and the data of the natural maps \eqref{eqn:Kunneth} and \eqref{eqn:Bocksteinmaps} when restricted to the K-theory groups with coefficients of this specific form.
Before we state this more precisely, we introduce some notation.
We denote by $F_P\underline{K}(A)$ the collection of abelian groups $K_i(A)$ and $K_i(A,\bZ_{p^m})$ for $m\in \bN$, $p$ a prime number and $i=0,1$.
We denote by $\Hom_{\Lambda}(F_P\underline{K}(A),F_P\underline{K}(B))$ the abelian group consisting of families of group homomorphisms
 \[
 \alpha^i:K_i(A)\rightarrow K_i(B)\quad \text{and}\quad \alpha_{p^m}^i:K_i(A;\bZ_{p^m})\rightarrow K_i(B;\bZ_{p^m})
 \]
satisfying
 \begin{align*}
     \kappa^i_{B,p^m,p^{m+1}}\circ \alpha^i_{p^{m+1}}&=\alpha^i_{p^{m}}\circ \kappa^i_{A,p^m,p^{m+1}}\\
     \kappa^i_{B,p^{m+1},p^m}\circ \alpha^i_{p^m}&=\alpha^i_{p^{m+1}}\circ \kappa^i_{A,p^{m+1},p^m}\\
     \rho_{B,p^m}^i\circ \alpha^i&=\alpha_{p^m}^i\circ \rho_{A,p^m}^i\\
     \beta_{B,p^m}^i\circ \alpha^i&=\alpha_{p^m}^i\circ \beta_{A,p^m}^i
 \end{align*}
which we can compose componentwise.
For $\alpha\in \Hom_{\Lambda}(\underline{K}(A),\underline{K}(B))$ we denote by $\alpha|_{F_P\underline{K}}$ the induced map in $\Hom_{\Lambda}(F_P\underline{K}(A),F_P\underline{K}(B))$ by forgetting the coefficients which are not a prime power.
When $A$ satisfies the UCT and $B$ is $\sigma$-unital this forgetful map is an isomorphism.
By picking the set theoretic inverses of this forgetful map mentioned above we get the following.

\begin{lemma}[{cf. \cite[Proposition 2.10]{DALO96}}]\label{lem:primesimplification}
Suppose $A$ is a separable C$^*$-algebra satisfying the UCT and $B$ is a $\sigma$-unital C$^*$-algebra.
There exists isomorphisms
\[
\Phi_{A,B}:\Hom_{\Lambda}(F_P\underline{K}(A), F_{P}\underline{K}(B))\rightarrow \Hom_{\Lambda}(\underline{K}(A),\underline{K}(B))
\]
such that 
\[
\Phi_{A,B}(\alpha|_{F_P\underline{K}})=\alpha,\quad \alpha\in \Hom_{\Lambda}(\underline{K}(A),\underline{K}(B)).
\]
Furthermore, if $B$ is also separable and satisfies the UCT and $C$ is another $\sigma$-unital C$^*$-algebra, then one has for all $\phi\in \Hom_{\Lambda}(F_P\underline{K}(A), F_{P}\underline{K}(B))$ and $\psi\in \Hom_{\Lambda}(F_P\underline{K}(B), F_{P}\underline{K}(C))$
that
\[
\Phi_{A,C}(\psi\circ\phi)=\Phi_{B,C}(\psi)\circ \Phi_{A,B}(\phi).
\]
\end{lemma}

To detect approximate unitary equivalence classes among morphisms between more general simple, nuclear C$^*$-algebras, one needs to enlarge the proposed classification invariant.
Recall from \cite{CLASS} the functor $\underline{K}T_u$ whose domain category is the category of unital C$^*$-algebras.
For a unital C$^*$-algebra $A$, $\underline{K}T_u(A)$ consists of $\underline{K}(A)$, the trace space of $A$ and the Hausdorffised unitary algebraic $K_1$-group (see \cite[Definition 2.6]{CLASS}).
A morphism $\underline{K}T_u(A)\rightarrow \underline{K}T_u(B)$ consists of a family of compatible maps between each of the components.
This is spelled out in \cite[Definition 3.5]{CLASS}.
We note that any unital $*$-homomorphism $\varphi:A\rightarrow B$ induces a morphism $\underline{K}T_u(\varphi):\underline{K}T_u(A)\rightarrow \underline{K}T_u(B)$, which assembles into a functor.
When $A$ is a unital Kirchberg algebra then $\underline{K}T_u(A)$ is simply $\underline{K}(A)$ and a $\underline{K}T_u$ morphism coincides with the notion of a $\Lambda$-morphism.
We may now state the classification of morphisms between unital, simple, separable, nuclear, $\Z$-stable C$^*$-algebras that satisfy the UCT. 

\begin{theorem}[cf.\ {\cite[Theorem B]{CLASS}}] \label{thm:KTU}
Let $A$ and $B$ be unital, simple, separable, nuclear $\Z$-stable C$^*$-algebras in the UCT class. Then:
\begin{enumerate}[leftmargin=*,label=\textup{(\arabic*)}]
\item For any two unital $*$-homomorphisms $\varphi,\psi:A\rightarrow B$, one has $\varphi\approx_{\mathrm{u}}\psi$ if and only if $\underline{K}T_u(\psi)=\underline{K}T_u(\varphi)$. \label{uniquenessKTU}
\item Given any morphism $\Phi:\underline{K}T_u(A)\rightarrow \underline{K}T_u(B)$, there exists a unital $*$-homomorphism $\varphi:A\rightarrow B$ with $\underline{K}T_u(\varphi)=\Phi$.
Moreover, if $\Phi$ is an isomorphism, then $\varphi$ can be chosen to be an isomorphism. \label{existenceKTU}
\end{enumerate}
\end{theorem}

\section{Equivariant (approximate) Murray--von Neumann equivalence} \label{sec:MvNequivalence}

In \cite{GA20}, Gabe introduces the notions of approximate and asymptotic Murray von Neumann equivalence for $*$-homomorphisms, which are crucial for his later classification of $\cO_2$-stable $*$-homomorphisms.
In this section, we introduce an equivariant version of approximate and asymptotic Murray von Neumann equivalence for $G$-coherent morphisms.
In analogy with \cite[Proposition 3.13]{GA20}, we show that when the target C$^*$-algebra is stable, these equivalence relations coincide with proper appproximate and proper asymptotic unitary equivalence.

\begin{definition}\label{defn:MvN}
Let $\omega\in Z^3(G,\bT)$ and let $\phi,\psi:(A,\alpha,\fu)\rightarrow (B,\beta,\fv)$ be $G$-coherent morphisms between $\omega$-anomalous actions of a countable discrete group $G$.
Then $\phi$ and $\psi$ are called \emph{approximately Murray von Neumann equivalent}, denoted $\phi\approx_{\mathrm{MvN}}^G \psi$, if for all $\varepsilon>0$, finite sets $K\subset G$, $F\subset A$ there exists $v\in B$ such that
    \[ 
    \lVert\beta_g(v^*)\phi_g(a)v-\psi_g(a)\rVert< \varepsilon,\quad \lVert\beta_g(v)\psi_g(a)v^*-\phi_g(a)\rVert< \varepsilon,
    \]
for all $g\in K$ and $a\in F$.

Suppose that $A$ is separable.
Then $\phi$ and $\psi$ are called \emph{asymptotically Murray von Neumann equivalent} if there exists a continuous path $v:[0,\infty)\rightarrow B$ such that
    \[
    \lim_{t\rightarrow\infty}\beta_g(v_t^*)\phi_g(a)v_t=\psi_g(a),\quad \lim_{t\rightarrow\infty}\beta_g(v_t)\psi_g(a)v_t^*=\phi_g(a)
    \]
for all $g\in G$ and $a\in A$.
\end{definition}

It is straightforward to see that our notions above coincide with \cite[Definition 3.4]{GA20} when $G$ is the trivial group.
Following an analogous argument to \cite[Lemma 3.5]{GA20}, we can always take $v$ or the path $(v_t)_{t\geq 0}$ in \cref{defn:MvN} to be contractive.
We will often assume that $A$ is separable in which case approximate Murray von Neumann equivalence coincides with the existence of a sequence $v_n\in B$ such that
\[
\lim_{n\rightarrow\infty}\beta_g(v_n^*)\phi_g(a)v_n=\psi_g(a),\quad \lim_{n\rightarrow\infty}\beta_g(v_n)\psi_g(a)v_n^*=\phi_g(a),
\]
for all $a\in A$ and $g\in G$.
\begin{remark}\label{rmk:unitalapproxmvn}
    Let $A$ and $B$ be unital, separable C$^*$-algebras and $\phi,\psi:(A,\alpha,\fu)\rightarrow (B,\beta,\fv)$ $G$-coherent morphisms between $\omega$-anomalous actions for some countable discrete group $G$ and $\omega\in Z^3(G,\bT)$. If $\phi_1$ and $\psi_1$ are unital maps, then $\phi\approx_{\mathrm{MvN}}^G\psi$ if and only if $\phi\approx_{\mathrm{u}}\psi$. Indeed, for the non-tivial direction, suppose that $\phi\approx_{\mathrm{MvN}}^G\psi$ and $(v_n)_{n\in \bN}\in B$ such that
    \[
\lim_{n\rightarrow\infty}\beta_g(v_n^*)\phi_g(a)v_n=\psi_g(a),\quad \lim_{n\rightarrow\infty}\beta_g(v_n)\psi_g(a)v_n^*=\phi_g(a).
\]
Letting $g=1$ and $a=\eins$ above we have that
\[
\lim_{n\rightarrow \infty}v_n^*v_n=\eins=\lim_{n\rightarrow \infty}v_nv_n^*.
\]
Thus we may find a sequence $w_n\in \cU(B)$ such that
\[
\lim_{n\rightarrow\infty}\lVert v_n-w_n\rVert=0.
\]
this sequence $w_n$ witnesses $\phi\approx_{\mathrm{u}}\psi$.
\end{remark}
Approximate or asymptotic Murray von Neumann equivalence is in practice much easier to access than proper approximate or asymptotic unitary equivalence respectively.
It will be important that under stability of the codomain algebra, these notions actually coincide.
Before we show this we need a technical lemma that will be useful to us in various instances.
We start with two preparatory lemmas.

\begin{lemma}\label{lem:unitaryeq}
Let $\cH$ be a separable Hilbert space and $S_1,S_2\in \cB(\cH)$ isometries such that $\eins-S_1S_1^*$ and $\eins-S_2S_2^*$ are infinite projections.
Then there exists a unitary $U\in \cB(\cH)$ such that $US_1=S_2$. 
\end{lemma}
\begin{proof}
As $\eins-S_jS_j^*$ are infinite for $j=1,2$, there exist isometries $R_j\in \cB(\cH)$ such that
\[
\eins-S_jS_j^*=R_jR_j^*,\quad j=1,2.
\]
Then $U=S_2S_1^*+R_2R_1^{*}$ is a unitary such that $US_1=S_2$.
\end{proof}
\begin{lemma}\label{lem:pathsofisometries}
Suppose $B$ is stable.
Then there is a norm-continuous path of isometries $s_t^{(j)}:[0,\infty)\rightarrow \cM(B)$ for $j=0,1$ such that
\[
\forall t\in [0,\infty):\ s_t^{(0)}{s_t^{(0)*}}+s_t^{(1)}{s_t^{(1)*}}=\eins,\quad\text{and}\quad  \lim_{t\rightarrow \infty}s_t^{(0)}= \eins\quad \text{strictly}.
\]
\end{lemma}
\begin{proof}
Let $\cH$ be a separable infinite-dimensional Hilbert space.
By stability, we have a unital and strictly continuous inclusion $\cB(\cH)=\cM(\mathbb{K})\subseteq\cM(B)$, so we may as well just consider the case $B=\mathbb{K}$.
For a chosen orthonormal basis $\{\delta_n\}_{n\in\mathbb N}$ of $\cH$, we shall subsequently identify $\mathbb{C}^\ell\cong\operatorname{span}\{ \delta_n \}_{n=1}^\ell\subset\cH$ for every $\ell\geq 1$.
Let us denote by $\{ e_{k,\ell} : k,\ell\geq 1\}$ the set of matrix units associated to this orthonormal basis.

We construct the paths $s_t^{(j)}$ for $j=0,1$ inductively on intervals $[n,n+1]$ for $n\geq 2$ and will obtain the desired paths defined on $[2,\infty)$ for notational convenience.
For the induction step, let 
\[
s_2^{(0)}=e_{1,1}+e_{2,2}+\sum_{k=3}^\infty e_{2k,k}\quad\text{and}\quad r_2=\sum_{k=3}^\infty e_{2k,k}.
\] 
As $\eins-s_2^{(0)}{s_2^{(0)*}}$ is an infinite projection in $\cB(\cH)$, it is Murray von Neumann equivalent to $\eins$.
Thus there exists an isometry $s_2^{(1)}$ such that $s_2^{(1)}{s_2^{(1)*}}=\eins-s_2^{(0)}{s_2^{(0)*}}$.
Let
\[
t_2=\sum_{k=3}^\infty e_{2k-3,k}=e_{3,3}+\sum_{k=4}^\infty e_{2k-3,k}.
\]
Then $r_2$ and $t_2$ are isometries in $\cB(\cH\ominus\bC^2)$ such that $\eins-r_2r_2^*$ and $\eins-t_2t_2^*$ are infinite.
By \cref{lem:unitaryeq}, there exists a unitary $V_2\in \cB(\cH\ominus\bC^2)$ such that $V_2r_2=t_2$.
Moreover, as $\cU\cB(\cH\ominus\bC^2)$ is path-connected \cite{CUHI87}, we pick a norm-continuous path $V_2: [0,1]\to \cU\cB(\cH\ominus\bC^2)$ with $V_2(0)=\eins$ and $V_2(1)=V_2$.
Let $U_2(t)=\id_{\bC^2}\oplus V_2(t)$ and 
\[
s_t^{(j)}= U_2(t-2)s_2^{(j)},\quad j=0,1,\ t\in [2,3].
\]
We can easily see that $s_t^{(0)}{s_t^{(0)*}}+s_t^{(1)}{s_t^{(1)*}}=\eins$ holds for all $t\in [2,3]$, and that $s_3^{(0)}=e_{1,1}+e_{2,2}+e_{3,3}+\sum_{k\geq 4} e_{2k-3,k}$.

As the induction hypothesis, assume $n\geq 3$ is given and the elements $s_t^{(j)}$ for $j=0,1$ have been defined for $t\in [2,n]$, and that we have
\[
s_n^{(0)}=e_{1,1}+e_{2,2}+\ldots e_{n,n}+\sum_{k=n+1}^\infty e_{2k-n,k}.
\]
For the induction step, consider the isometries $r_n=\sum_{k\geq n+1} e_{2k-n,k}$ and $t_n=\sum_{k\geq n+1} e_{2k-(n+1),k}$ in $\cB(\cH\ominus\bC^n)$.
Applying \cref{lem:unitaryeq}, we may find a unitary $V_n\in \cU\cB(\cH\ominus\bC^n)$ such that $V_nr_n=t_n$.
Letting $V_n: [0,1]\to \cU\cB(\cH\ominus\bC^n)$ be a continuous path such that $V_n(0)=\eins$ and $V_n(1)=V_n$, set $U_n(t)=\id_{\bC^n}\oplus V_n(t)$ and
\[
s_t^{(j)}=U_n(t-n)s_n^{(j)},\ j=0,1,\ t\in [n,n+1].
\]
This construction yields two paths of isometries $s^{(j)}:[2,\infty)\rightarrow\cB(\cH)$ satisfying $s_t^{(0)}{s_t^{(0)*}}+s_t^{(1)}{s_t^{(1)*}}=\eins$ for all $t\geq 2$.
Furthermore, we can read off the construction that $s_t^{(0)}\delta_k=\delta_k$ for every $k\in\mathbb N$ and $t\geq k$.
This implies $s_t^{(0)}\rightarrow \eins$ in the strict topology and finishes the proof.
\end{proof}

\begin{lemma}\label{lem:approxleftmult}
Let $B$ be a stable \cstar-algebra and $v:[0,\infty)\rightarrow B_1$ a continuous path of contractions.
Then there exists a continuous unitary path $w:[0,\infty)\rightarrow \cU(\eins+B)$ with $w_0=\eins$ such that
\[
\limsup_{t\rightarrow \infty}\lVert (w_t-v_t)e\rVert^2 \leq \limsup_{t\rightarrow\infty} \lVert e^*(\eins-v_t^*v_t)e\rVert
\]
for all $e\in B$.
\end{lemma}
\begin{proof}
As $B$ is stable, the unitary group $\cU\cM(B)$ is path-connected \cite{CUHI87}.
Thus, any norm continuous path $W:[0,\infty)\rightarrow \cU\cM(B)$ can be extended (by adding an initial line segment) to a norm-continuous path starting at $\eins$.
As a consequence of \cite[Lemma 4.3]{GASZ25}, there exists a norm-continuous path $w:[0,\infty)\rightarrow \cU(\eins+B)$ with $w_0=\eins$ such that
\[
\limsup_{t\rightarrow\infty} \lVert (w_t-W_t)e\rVert=0,\quad e\in B.
\]
Therefore, it suffices to find a path $w$ as in the statement of the lemma with values in $\cU\cM(B)$ instead.
 
Let $s_t^{(j)}$ for $j=0,1$ be norm-continuous paths of isometries satisfying the conclusion of Lemma \ref{lem:pathsofisometries}.
By cutting away an initial line segment and reparametrizing these paths on $[0,\infty)$, if necessary, we may assume that
\[
\lim_{t\rightarrow\infty}\lVert s_t^{(0)}v_t-v_t \rVert= 0.
\]
For $t\geq 0$, let $r_t^{(0)}=s_t^{(0)}$, $r_t^{(1)}=s_t^{(1)}s_t^{(0)}$ and $r_t^{(2)}=(s_t^{(1)})^2$. 
Then $r^{(j)}$, for $j=0,1,2$, define norm-continuous paths of isometries such that
\[
\sum_{j=0}^2 r_t^{(j)}{r_t^{(j)*}}=\eins,\quad r_t^{(0)}\rightarrow \eins\quad \text{strictly},\quad \lim_{t\rightarrow\infty}\lVert r_t^{(0)}v_t-v_t\rVert=0.
\]
We let
\[
R_t=r_t^{(0)}v_t{r_t^{(0)*}}+r_t^{(1)}(\eins-v_t^*v_t)^{1/2}{r_t^{(0)*}},\quad t\geq 0,
\]
which induces an element $R$ in the \cstar-algebra $C_b([0,\infty),\cM(B))$.
Note that
\[
R_tR_t^*\leq \eins-{r_t^{(2)*}}r_t^{(2)},\quad 
R_t^*R_t=r_t^{(0)}{r_t^{(0)*}},
\]
so $r^{(2)}{r^{(2)*}}\leq \eins-RR^*$ and $r^{(2)}{r^{(2)*}}\leq \eins-R^*R$.
In particular $\eins-RR^*$ and $\eins-R^*R$ are properly infinite and full projections in $C_b([0,\infty),\cM(B))$.
As a consequence of \cite{CU81} (see also \cite[Exercise 4.9]{IntroK}) there is a partial isometry $V\in C_b([0,\infty),\cM(B))$ such that
\[
VV^*=\eins-RR^*,\quad V^*V=\eins-R^*R=\eins-r^{(0)}{r^{(0)*}}.
\]
Now, let $w=R+V\in C_b([0,\infty),\cM(B))$.
Then we have that
    \begin{align*}
    w^*w=(R+V)^*(R+V)&=R^*R+V^*R+V^*V+RV^*\\
    &=\eins+V^*VV^*RR^*R+RR^*RV^*VV^* = \eins
    \end{align*}
and similarly $ww^*=\eins$.
Thus $w$ is a norm-continuous path valued in $\cU\cM(B)$.
Moreover, we observe for every $e\in B$ that
    \begin{align*}
    &\limsup_{t\rightarrow\infty}\lVert(w_t-v_t)e\rVert^2 \\
    &\leq \limsup_{t\rightarrow\infty} \big( \underbrace{\lVert (r_t^{(0)}v_t{r_t^{(0)*}}-v_t)e\rVert}_{\to 0} + \lVert r_t^{(1)}(\eins-v_t^*v_t)^{1/2}{r_t^{(0)*}}e\rVert + \lVert V_t e\rVert \big)^2
    \\
    &\leq \limsup_{t\rightarrow\infty} \big( \lVert (\eins-v_t^*v_t)^{1/2}e\rVert+\lVert V_t e\rVert \big)^2\\
    &= \limsup_{t\rightarrow\infty} \big( \lVert (\eins-v_t^*v_t)^{1/2}e\rVert+\underbrace{\lVert V_t (\eins-r_t^{(0)}{r_t^{(0)*}}) e\rVert}_{\to 0} \big)^2\\
    &\leq\limsup_{t\rightarrow\infty} \lVert (\eins-v_t^*v_t)^{1/2}e\rVert^2 \ = \ \limsup_{t\rightarrow\infty} \lVert e^*(\eins-v_t^*v_t)e\rVert,
    \end{align*}
as required.
\end{proof}

Recall from \cite{RO16} that a C$^*$-algebra $B$ is said to have \emph{almost stable rank one} if $B\subset \overline{GL(B^\dagger)}$.

\begin{lemma} \label{lem:sequencesettingLemma}
Let $B$ be a \cstar-algebra that is either stable or has almost stable rank one, and let $v\in B_1$.
For any finite set $F\subset B$ and $\varepsilon>0$, there exists $w\in \cU(\eins+B)$ such that $\lVert (w-v)e\rVert^2 < \varepsilon+\lVert e^*(\eins-v^*v)e\rVert$ for $e\in F$.
\end{lemma}
\begin{proof}
   Suppose $B$ is a stable \cstar-algebra and $v\in B_1$.
Then applying Lemma \ref{lem:approxleftmult} with the constant path $v$, we get a continuous unitary path $w: [0,\infty)\to\cU(\eins+B)$ such that 
\[
\limsup_{t\rightarrow \infty}\lVert (w_t-v)e\rVert^2 \leq \lVert e^*(\eins-v^*v)e\rVert
\]
for all $e\in B$.
This implies the result when $B$ is stable.

Suppose $B$ has almost stable rank one.
Let $\delta>0$ such that $\delta+\sqrt{2\delta}\leq\varepsilon$.
Choose a contraction $y\in GL(A^\dagger)$ with 
\[
\lVert y-v\rVert <\frac{\delta}{1+\max_{x\in F\cup F^2} \lVert x\rVert}
\]
Set $w=y\lvert y\rvert^{-1}$.
We first note that $\|y^*y-v^*v\|<\frac{2\delta}{1+\max_{x\in F\cup F^2} \lVert x\rVert}\leq 2\delta$ follows from the triangle inequality.
Then for $e\in F$ we have that
\begin{align*}
    \lVert (w-v)e\rVert&\leq \lVert (y-v)e\rVert+\lVert w(\eins-\lvert y\rvert)e\rVert\\
    &< \delta+\sqrt{\lVert e^*(\eins-\lvert y\rvert)^2e\rVert}\\
    &\leq \delta+\sqrt{\lVert e^*(\eins-y^*y)e\rVert}\\
    &\leq \delta +\sqrt{2\delta}+\sqrt{\lVert e^*(\eins-v^*v)e\rVert}\\
    &\leq \varepsilon +\sqrt{\lVert e^*(\eins-v^*v)e\rVert}.
\end{align*}
Here we have employed functional calculus with the inequality $(1-t)^2\leq 1-t^2$ for all $t\in [0,1]$ in the third inequality.
\end{proof}

We shall prove the dynamical analogue of \cite[Proposition 3.13]{GA20}.
Even without the dynamics, our proof somewhat differs from the one in \cite{GA20}.

\begin{proposition}\label{prop:stablemvn}
Let $A$ and $B$ be \cstar-algebras with $A$ separable.
Let $\phi,\psi:(A,\alpha,\fu)\rightarrow (B,\beta,\fv)$ be two $G$-coherent morphisms between anomalous actions of a countable discrete group $G$.
Then:
	\begin{enumerate}[leftmargin=*,label=\textup{(\roman*)}]
	\item If $B$ is either stable or has almost stable rank one, then $\phi$ and $\psi$ are approximately Murray von Neumann equivalent if and only if they are proper approximately unitarily equivalent. 
    \item If $B$ is stable, then $\phi$ and $\psi$ are asymptotically Murray von Neumann equivalent if and only if they are proper asymptotically unitarily equivalent.
    \end{enumerate}
\end{proposition}
\begin{proof}
As the ``if'' parts are tautological, we show the ``only if'' parts.
First observe that if $a\in A$ and $b\in B$, then by Proposition \ref{prop:leftinnerproduct} one has
\begin{align}
\lVert \beta_g(b)\phi_g(a)\rVert^2&=\lVert \beta_g(b)\phi_g(a)\phi_g(a)^*\beta_g(b)^*\rVert=\lVert b\phi_1(\alpha_g^{-1}(aa^*))b^*\rVert\label{eqn:leftnorm}\\
&=\lVert b\phi_1(\alpha_g^{-1}(a))\rVert^2 \notag
\end{align}
and
\begin{equation}\label{eqn:rightnorm}
    \lVert \phi_g(a)b\rVert^2=\lVert b^*\phi_g(a)^*\phi_g(a)b\rVert=\lVert b^*\phi_1(a^*a)b\rVert=\lVert \phi_1(a)b\rVert ^2.
\end{equation}
Now, suppose $\phi$ and $\psi$ are approximately Murray von Neumann equivalent and $B$ is stable or has almost stable rank one.
There exists a sequence $v_n\in B_1$ such that
\[
\lim_{n\rightarrow\infty}\beta_g(v_n^*)\phi_g(a)v_n=\psi_g(a),\quad \lim_{n\rightarrow\infty}\beta_g(v_n)\psi_g(a)v_n^*=\phi_g(a),
\]
for all $a\in A$ and $g\in G$.
By \cite[Lemma 3.8]{GA20}, it follows that $(\eins-v_n^* v_n) \psi_1(a) \to 0$ for all $a\in A$.
Then as a consequence of Lemma \ref{lem:sequencesettingLemma}, one may find a sequence $w_n\in \cU(\eins+B)$ such that
\[
\lim_{n\rightarrow\infty}(w_n-v_n)\psi_1(a)=0,\quad a\in A.
\]
In particular, given $a\in A$, it follows from \eqref{eqn:leftnorm} and \eqref{eqn:rightnorm} that
\begin{align*}
    &\lVert\beta_g(w_n)\psi_g(a)w_n^*-\phi_g(a)\rVert \\
    &\leq \lVert\beta_g(w_n-v_n)\psi_g(a)w_n^*\rVert+\lVert\beta_g(v_n)\psi_g(a)(w_n^*-v_n^*)\rVert + \lVert\beta_g(v_n)\psi_g(a)v_n^*-\phi_g(a)\rVert\\
    &\leq \lVert\beta_g(w_n-v_n)\psi_g(a)\rVert+\lVert\psi_g(a)(w_n^*-v_n^*)\rVert + \lVert\beta_g(v_n)\psi_g(a)v_n^*-\phi_g(a)\rVert\\
    &=\lVert (w_n-v_n)\psi_1(\alpha_g^{-1}(a))\rVert+\lVert \psi_1(a)(w_n^*-v_n^*)\rVert + \lVert\beta_g(v_n)\psi_g(a)v_n^*-\phi_g(a)\rVert\\
   &\xlongrightarrow{n\to\infty} 0.
\end{align*}
This implies that $\psi$ and $\phi$ are proper approximate unitarily equivalent.

Now suppose $\phi$ and $\psi$ are asymptotically unitarily equivalent and $B$ is stable. Let $v_t:[0,\infty)\rightarrow B_1$ be a continuous paths such that
\[
\lim_{t\rightarrow\infty}\beta_g(v_t^*)\phi_g(a)v_t=\psi_g(a),\quad \lim_{t\rightarrow\infty} \beta_g(v_t)\psi_g(a)v_t^*=\phi_g(a)
\]
for all $a\in A$ and $g\in G$.
As above, \cite[Lemma 3.8]{GA20} entails that $(\eins-v_t^* v_t) \psi_1(a) \to 0$ for all $a\in A$. 
We apply \cref{lem:approxleftmult} to find a unitary path $w: [0,\infty)\to\cU(\eins+B)$ satisfying the conclusion stated there. 
Then by \eqref{eqn:leftnorm} and \eqref{eqn:rightnorm} we have for all $a\in A$ and $g\in G$ that
\begin{align*}
& \lVert \beta_g(w_t)\psi_g(a)w_t^*-\phi_g(a)\rVert \\
&\leq \lVert\beta_g(w_t-v_t)\psi_g(a)w_t^*\rVert+\lVert\beta_g(v_t)\psi_g(a)(w_t^*-v_t^*)\rVert + \lVert\beta_g(v_t)\psi_g(a)v_t^*-\phi_g(a)\rVert\\
    &= \lVert (w_t-v_t)\psi_1(\alpha_g^{-1}(a))\rVert+\lVert\psi_1(a)(w_t^*-v_t^*)\rVert + \lVert\beta_g(v_t)\psi_g(a)v_t^*-\phi_g(a)\rVert.
\end{align*}
Note that the limsup of the right-hand side is seen to vanish.
Therefore
\[
\lim_{t\rightarrow\infty}\beta_g(w_t)\psi_g(a)w_t^*=\phi_g(a),\quad a\in A,
\]
which means that $\psi$ and $\phi$ are proper asymptotically unitarily equivalent.
\end{proof}

\section{Existence and Uniqueness}\label{sec:exist+uniq}

In this section we prove existence and uniqueness results for $G$-coherent morphisms in the spirit of \cite[Section 3.1]{GASA16}, under the assumption that $G$ is finite and the anomalous action belonging to the codomain has the Rokhlin property.
These results will form the basis of Section \ref{sec:classification} to classify Rokhlin anomalous actions.

\begin{proposition}\label{prop:uniqueness}
Let $(A,\alpha,\fu)$ and $(B,\beta,\fv)$ be anomalous actions of a finite group $G$ on separable C$^*$-algebras such that $(B,\beta,\fv)$ has the Rokhlin property.
Let $\phi,\psi:(A,\alpha,\fu)\rightarrow (B,\beta,\fv)$ be two $G$-coherent morphisms.
If $\phi_1\approx_{\mathrm{MvN}}\psi_1$, then $\phi\approx_{\mathrm{MvN}}^G \psi$.
\end{proposition}
\begin{proof}
Let $D\subseteq\cM(A)$ be the \cstar-algebra generated by $\{\fu_{g,h}:g,h\in G\}$.
Note that $D$ is $\alpha$-invariant due to the cocycle identity for $\fu$.
Let $e_n$ be an approximate unit of $A$ which is quasicentral relative to $D$.
Then $e_n'=\sum_{g\in G}\alpha_g(e_n)$ defines an approximate unit for $A$ which is still quasicentral relative to $D$ and such that $\|\alpha_g(e_n')-e_n'\|\rightarrow 0$ as $n\to\infty$.
Let $\varepsilon>0$ and $\cF\subset A_{1}$ be a finite set.
Thus by passing to large enough $n$ we may pick $e\in A_1^{+}$ such that
    \begin{enumerate}[label=\textup{(\alph*)}]
    \item  $\|\alpha_g(e)-e\|<\frac{\varepsilon}{2|G|^2}$, \label{prop:uniqueness:proof-a}
    \item  $\|ae-a\|< \frac{\varepsilon}{|G|^2}$, \label{prop:uniqueness:proof-b}
    \end{enumerate}
for all $g\in G$ and $a\in\cF$.
As $\phi$ and $\psi$ are approximately Murray von Neumann equivalent, there exists an element $s\in B_1$ such that 
\begin{equation} \label{prop:uniqueness:proof-s} 
s\phi_1(a)s^*=_{\varepsilon/|G|}\psi_1(a),\quad s^*\psi_1(a)s=_{\varepsilon/|G|} \phi_1(a)
\end{equation}
whenever $a\in \alpha_{gh}^{-1}(\fu_{g,h}^*\cF)$ for $g,h\in G$. 
Using the Rokhlin property via Lemma~\ref{lem:Rp} for $(B,\beta,\fv)$ we can pick a positive contraction $r\in B_{1}^+$ such that
    \begin{enumerate}[label=\textup{(\alph*)},resume]
    \item $\lVert\beta_{g}(r)b-b \beta_g(r)\rVert < \frac{\varepsilon}{|G|^2}$,\label{prop:uniqueness:proof-c}
    \item $\lVert\beta_g(r)\beta_h(r)b-\delta_{g,h}\beta_{g}(r)b\rVert < \frac{\varepsilon}{|G|^2}$, \label{prop:uniqueness:proof-d}
    \item $\lVert \sum_{g\in G}\beta_g(r)b-b\lVert < \varepsilon$, \label{prop:uniqueness:proof-e}
    \end{enumerate}
for all $g,h\in G$ and all 
    \[
    b\in \bigcup_{g,h\in G,\Phi\in \{\phi,\psi\}} \Phi_{gh}(\fu_{g,h}^*\cF) \bigcup_{g\in G,\Phi\in \{\phi,\psi\}}\{\Phi_g(e)\}\cup \{s,s^*\}.
    \]
Now, we let
    \[
    v=\sum_{g\in G}\psi_g(e)^*\beta_g(sr)\phi_g(e)\in B.
    \]
For every $a\in \cF$, we observe
{\small
    \begin{align*}
    &\beta_g(v)\phi_g(a)v^*\\
    &\quad =\sum_{h\in G}\beta_g(\psi_h(e)^*)\fv_{g,h}\beta_{gh}(sr)\fv_{g,h}^*\beta_g(\phi_h(e))\phi_g(a)v^*\\
    &\quad =\sum_{h\in G} \beta_g(\psi_h(e)^*)\fv_{g,h}\beta_{gh}(sr) \phi_{gh}\big( \fu_{g,h}^*\alpha_g(e)a \big) v^*\\
    &\stackrel{\ref{prop:uniqueness:proof-a},\ref{prop:uniqueness:proof-b}}{\quad =_{2\varepsilon}} \sum_{h\in G} \beta_g(\psi_h(e)^*)\fv_{g,h}\beta_{gh}(sr)\phi_{gh}(\fu_{g,h}^*a)v^*\\
    & \quad = \sum_{h,f\in G} \beta_g(\psi_h(e)^*)\fv_{g,h}\beta_{gh}(sr)\phi_{gh}(\fu_{g,h}^*a) \phi_f(e)^* \beta_f(sr)^* \psi_f(e) \\
    &\stackrel{ \ref{prop:uniqueness:proof-c}}{\quad =_{2\varepsilon}} \sum_{h,f\in G} \beta_g(\psi_h(e)^*)\fv_{g,h}\beta_{gh}(s)\phi_{gh}(\fu_{g,h}^*a) \phi_f(e)^* \beta_{gh}(r)\beta_f(r) \beta_f(s)^* \psi_f(e) \\
    &\stackrel{\ref{prop:uniqueness:proof-c},\ref{prop:uniqueness:proof-d}}{\quad =_{3\varepsilon}}\sum_{h\in G} \beta_g(\psi_h(e)^*)\fv_{g,h}\beta_{gh}(s)\phi_{gh}(\fu_{g,h}^*a)\phi_{gh}(e)^*\beta_{gh}(s^*)\psi_{gh}(e) \beta_{gh}(r)\\
    &\stackrel{\ \eqref{eqn:leftinnerproduct}}{\quad =_{\ }}  \sum_{h\in G}\beta_g(\psi_h(e)^*)\fv_{g,h}\beta_{gh}(s)\beta_{gh}\phi_{1} \alpha_{gh}^{-1}(\fu_{g,h}^*ae) \beta_{gh}(s^*)\psi_{gh}(e)\beta_{gh}(r)\\
    &\stackrel{\ \ref{prop:uniqueness:proof-b}}{\quad =_{\varepsilon}} \sum_{h\in G}\beta_g(\psi_h(e)^*)\fv_{g,h}\beta_{gh}(s)\beta_{gh}\phi_{1} \alpha_{gh}^{-1}(\fu_{g,h}^*a) \beta_{gh}(s^*)\psi_{gh}(e)\beta_{gh}(r)\\
    &\stackrel{\ \eqref{prop:uniqueness:proof-s}}{\quad =_{\varepsilon}} \sum_{h\in G}\beta_g(\psi_h(e)^*)\fv_{g,h}\beta_{gh}\psi_{1}\alpha_{gh}^{-1}(\fu_{g,h}^*a)\psi_{gh}(e)\beta_{gh}(r)\\
    &\quad =  \sum_{h\in G}\beta_g(\psi_h(e)^*)\fv_{g,h}\psi_{gh}(\fu_{g,h}^*ae)\beta_{gh}(r)\\
    &\quad = \sum_{h\in G}\beta_g(\psi_h(e)^*)\beta_g\psi_h\alpha_g^{-1}(a)\psi_h(e)\beta_{gh}(r)\\
    &\quad =\sum_{h\in G}\beta_g\psi_1(e\alpha_g^{-1}(a))\psi_g(e)\beta_{gh}(r)\\
    &\stackrel{\ref{prop:uniqueness:proof-a},\ref{prop:uniqueness:proof-b}}{\quad =_{2\varepsilon}}\sum_{h\in G}\psi_g(ae)\beta_{gh}(r)\\
    &\stackrel{\ \ref{prop:uniqueness:proof-b}}{\quad =_{\varepsilon}}\sum_{h\in G}\psi_g(a)\beta_h(r)\\
    &\stackrel{\ \ref{prop:uniqueness:proof-e}}{\quad =_{\varepsilon}}\psi_g(a).
    \end{align*}
    }
Similarly, it follows from the symmetry of the above argument in swapping $\psi$ with $\phi$ and $s^*$ with $s$ that also 
    \[
    \beta_g(v^*)\psi_g(a)v=_{13\varepsilon}\phi_g(a)
    \]
for all $a\in \cF$. As $\varepsilon$ and $\cF$ are arbitrary the result follows.
\end{proof}
\begin{corollary} \label{cor: uniqueness}
Let $(A,\alpha,\fu)$ and $(B,\beta,\fv)$ be anomalous actions of a finite group $G$ on separable C$^*$-algebras such that $(B,\beta,\fv)$ has the Rokhlin property.
Let $\phi,\psi:(A,\alpha,\fu)\rightarrow (B,\beta,\fv)$ be two $G$-coherent morphisms.
Assume that either one of the following is true:
	\begin{enumerate}[leftmargin=*,label=\textup{(\roman*)}]
	\item $B$ is stable. \label{cor: uniqueness:1}
	\item $A$ and $B$ are unital and the $*$-homomorphisms $\phi_1, \psi_1$ are unital. \label{cor: uniqueness:2}
    \item $B$ has almost stable rank one. \label{cor: uniqueness:3}
	\end{enumerate}
Then $\phi_1\approx_{\mathrm{pu}}\psi_1$ if and only if $\phi\approx_{\mathrm{pu}}^G\psi$.
\end{corollary}
\begin{proof}
\ref{cor: uniqueness:1} and \ref{cor: uniqueness:3} follow as a direct consequence of Proposition \ref{prop:uniqueness} and \cref{prop:stablemvn}.

\ref{cor: uniqueness:2}: follows as a direct consequence of Proposition \ref{prop:uniqueness} and \cref{rmk:unitalapproxmvn}.
\end{proof}

It will be important for us to understand when there exists a $G$-coherent morphism between two anomalous actions with the Rokhlin property.
The following result does this for us.

\begin{theorem} \label{thm: existence}
Let $(A,\alpha,\fu)$ and $(B,\beta,\fv)$ be anomalous actions of a finite group $G$ on separable C$^*$-algebras such that $(B,\beta,\fv)$ has the Rokhlin property and $B$ is stable or has almost stable rank one.
Let $\varphi:A\rightarrow B$ be a $*$-homomorphism such that $\varphi\alpha_g\approx_{\mathrm{pu}} \beta_g \varphi$ for all $g\in G$ and $o(\alpha,\fu)=o(\beta,\fv)$.
Then there exists a $G$-coherent morphism $\phi:(A,\alpha,\fu)\rightarrow (B,\beta,\fv)$ such that $\phi_1\approx_{\mathrm{pu}}\varphi$.
\end{theorem}

To prove this result we require the following technical lemma similar in flavour to \cite[Lemma 3.3]{IZU04I} (see also \cite[Proposition 3.4]{GASA16}).

\begin{lemma} \label{lem:iterativelemma}
Let $A$ and $B$ be separable \cstar-algebras and assume $B$ is either stable or has almost stable rank one.
Let $(A,\alpha,\fu)$ and $(B,\beta,\fv)$ be anomalous actions of a finite group $G$ with $o(\alpha,\fu)=o(\beta,\fv)$ such that $\beta$ has the Rokhlin property.
Suppose $\varphi:A\rightarrow B$ is a $*$-homomorphism such that $\varphi\alpha_g\approx_{\mathrm{pu}}\beta_g\varphi$ for all $g\in G$, and let $\bv: G\to \cU(\eins+B)$ be a map.
Then for all $\varepsilon>0$ and $\cF\subset A_{1}$ there exists a map $\wv: G\to \cU(\eins+B)$ such that, if we write $\phi_g(a)=\bv_g\varphi(a)\bv_1^*$ and $\widetilde{\phi}_g(a)=\wv_g\varphi(a)\wv_1^*$, one has the following for all $g\in G$ and all $a,b,c\in\cF$:
\begin{equation} \label{equivarianceapprox}
\fv_{g,h}^*\beta_g(\wf_h(a))\wf_g(b)=_{\varepsilon}\wf_{gh}(\fu_{g,h}^*\alpha_g(a)b),
\end{equation}
\begin{equation} \label{leftinnerapprox}
\wf_g(a)\wf_g(b)^*=_{\varepsilon}\beta_g\wf_1\alpha_g^{-1}(ab^*),
\end{equation}
\begin{equation} \label{closenessapprox}
\begin{multlined} 
    	\begin{array}{cl}
        \multicolumn{2}{l}{ \| \wf_g(\alpha_g(a)bc)-\phi_g(\alpha_g(a)bc)\| } \\
        \leq & \varepsilon+\lVert\beta_g(\phi_1(a))\phi_g(b)-\phi_g(\alpha_g(a)b)\rVert \\
        +&\displaystyle \max_{k\in G} \Big( \lVert \fv_{g,g^{-1}k}^*\beta_g(\phi_{g^{-1}k}(a))\phi_g(b)-\phi_k(\fu_{g,g^{-1}k}^*\alpha_g(a)b)\rVert\\
        +& \lVert \beta_k\phi_1\alpha_k^{-1}(\fu_{g,g^{-1}k}^*\alpha_g(a)bc)-\phi_k(\fu_{g,g^{-1}k}^*\alpha_g(a)b)\phi_k(c^*)^*\rVert \Big).
    	\end{array}
\end{multlined}
\end{equation}
\end{lemma}
\begin{proof}
Let us assume $\varepsilon<1$.
Consider the finite sets
\[
\cV_A=\{\fu_{g,h}:g,h\in G\}\subset\cU\cM(A), \quad \cV_B=\{\fv_{g,h}:g,h\in G\}\subset\cU\cM(B).
\]
For the remainder of the proof we consider the maps $\phi_g$ for $g\in G$ as in the statement, which are contractive linear maps and $\phi_1$ is a $*$-homomorphism. 

Let $e_n$ be an idempotent approximate unit for $A$. 
Then we observe for all $a\in A$ and $g\in G$ that
\[
\phi_g(e_n)^*\phi_g(e_n)\phi_1(a)=\phi_1(e_n^2a)\rightarrow \phi_1(a).
\]
Thus, by Lemma  \ref{lem:sequencesettingLemma}, if $n\in \bN$ is sufficiently large, there exist unitaries $\bs_g\in \cU(\eins+B)$ for $g\in G$ such that
\begin{equation} \label{lem:iterativelemma:eq-1}
    e_{n}a=_{\frac{\varepsilon}{2|G|^2}}a,\quad \bs_g\phi_1(a)=_{\varepsilon/2}\phi_g(e_{n})\phi_1(a)=_{\varepsilon/2}\phi_g(a), 
\end{equation}
for all $a\in \bigcup_{g\in G}\alpha_g^{-1}(\cF)\cup \alpha_g^{-1}(\cF^*)\cup\ \cV_A\alpha_g(\cF)\cF$ and all $n\geq n_0$.
By hypothesis, there exist unitaries $\bw_g\in \cU(\eins+B)$ such that
\begin{equation} \label{lem:iterativelemma:eq-choice-bw}
    \Ad(\bw_g)\phi_1(a)=_{\varepsilon/|G|}\Ad(\bs_g^*)\beta_g\phi_1\alpha_g^{-1}(a)
\end{equation}
for all $g\in G$ and $a\in \bigcup_{g\in G, i=1,2} \cV_A^*\alpha_g(\cF\cF^*)\cV_A\cup \cV_A^*\cF\cup (\cU^*)^2\alpha_g(\cF)\cF^i$.

Using Lemma~\ref{lem:Rp}, let $f_n\in B_1$ be a sequence of positive contractions such that for all $b\in B$ and $g,h\in G$,
\begin{enumerate}[leftmargin=*,label=\textup{(\alph*')}]
    \item $\|f_nb-bf_n\|\to 0$,\label{enum:seqrokhlin1}
    \item $\|\beta_g(f_n)\beta_h(f_n)b-\delta_{g,h}\beta_g(f_n)b\|\to 0$,\label{enum:seqrokhlin2}
    \item $\|\sum_{k\in G}\beta_k(f_n)b-b\|\to 0$\label{enum:seqrokhlin3}.
\end{enumerate}
For each $n\geq n_0$, consider the contractions
\[
w_{g,n}=\sum_{k\in G}\beta_k(f_n)\beta_g(\bs_{g^{-1}k}^*)\fv_{g,g^{-1}k}\bs_k \bw_k \phi_1(\fu_{g,g^{-1}k}^*e_n).
\]
Then for all $a\in A$, $g\in G$ and $\delta>0$ one has a sufficiently large $n\in \bN$ such that
\begin{align*}
    &w_{g,n}^*w_{g,n}\phi_1(a)\\   
    &\quad\ \ =\sum_{k,k'\in G}\phi_1(e_n\fu_{g,g^{-1}k})\bw_k^*\bs_k^*\fv_{g,g^{-1}k}^*\beta_g(\bs_{g^{-1}k})\beta_k(f_n)\beta_{k'}(f_n)\beta_g(s_{g^{1}k'}^*)\fv_{g,g^{-1}k'}\\
    &\quad\quad\quad\quad\quad \bw_{k'}\phi_1(\fu_{g,g^{-1}k'}^*e_na)\\    
    &\quad\ \ =_{\delta}\sum_{k,k'\in G}\phi_1(e_n\fu_{g,g^{-1}k})\bw_k^*\bs_k^*\fv_{g,g^{-1}k}^*\beta_g(\bs_{g^{-1}k})\beta_k(f_n)\beta_{k'}(f_n)\beta_g(s_{g^{1}k'}^*)\fv_{g,g^{-1}k'}\\&\quad\quad\quad\quad\quad \bw_{k'}\phi_1(\fu_{g,g^{-1}k'}^*a)\\    &\stackrel{\ \ \ref{enum:seqrokhlin1},\ref{enum:seqrokhlin2}}{\quad =_{\delta}}\sum_{k\in G}\phi_1(e_na)\beta_k(f_n)\\
    &\quad\ \ =_\delta \sum_{k\in G}\phi_1(a)\beta_k(f_n)\\
    &\quad\ \stackrel{\ref{enum:seqrokhlin3}}{\ =_\delta}\phi_1(a).
\end{align*}
Thus, by Lemma \ref{lem:sequencesettingLemma} there is $n_1>n_0$ and a unitary $\wv_g\in \cU(\eins+B)$ such that
\[
w_{g,n_1}\phi_1(a)=_{\varepsilon/2|G|}\wv_g\phi_1(a)
\]
for all
\[
a\in \bigcup_{g\in G,i=1,2,j,l,t=0,1,2}(\cF^*)^t \cV_A^*\alpha_g(\cF^i(\cF^*)^j)\cV_A\cF^l.
\]
Here we set $\cF^0=\{\eins\}$ for convenience. 

Choose an element $e\in A_1^+$ such that
\[
\max\{\lVert d(\eins-e) \rVert,\lVert (\eins-e)d\rVert\}\leq \frac{\varepsilon}{2|G|}
\]
for all elements of the form
\[
d=\beta_g(\bs_k^*)\fv_{g,k}^* \bs_{gk} \bw_{gk}\phi_1(\fu_{g,k}^*\alpha_g(a)bc) \bw_{gk}^*-\phi_g(\alpha_g(a)bc)
\] 
for any $g,k\in G$ and $a,b,c\in \cF$. 
Without loss of generality we may also assume that $n_1$ is large enough such that
\begin{enumerate}[leftmargin=*,label=\textup{(\alph*)}]
    \item $\beta_g(f_{n_1})b=_{\varepsilon/|G|^2}b\beta_g(f_{n_1})$, \label{lem:iterativelemma:eq-a}
    \item $\beta_g(f_{n_1})\beta_h(f_{n_1})b=_{\varepsilon/|G|^2}\delta_{g,h}\beta_g(f_{n_1})b$, \label{lem:iterativelemma:eq-b}
    \item $\sum_{k\in G}\beta_k(f_{n_1})b =_{\varepsilon} b$, \label{lem:iterativelemma:eq-c}
\end{enumerate}
for all $g,h\in G$ and $b\in \cG$, where
\begin{align*}
    \cG=&\bigcup_{g,h,l\in G,i\in\{1,2\},j\in \{0,1,2\}}\cV_B^*\beta_g(\cV_B \bs_h \bw_h\phi_1(\cV_A^* \alpha_l^{-1}(\cF^i)\cF^j))\cV_B\\
    &\bigcup_{g,h\in G}\beta_g( \bs_h-\eins) \bigcup_{g,k\in G}\beta_g( \bw_k-\eins) \bigcup_{g\in G}\phi_g(\alpha_g(\cF)\cF^2)\cup \{e\}.
\end{align*}
Let 
\[
\wf_g(a)=\wv_g\phi_1(a)\wv_1^*,\quad g\in G.
\]
We observe for every 
\[
a\in \bigcup_{g\in G,i=1,2,j,l=0,1,2}(\cF^*)^t \cV_A^*\alpha_g(\cF^i(\cF^*)^j) \cV_A \cF^l \quad\text{and}\quad g\in G
\]
that
\begin{align} \label{eqn:twiddle}
    \wf_g(a)&=_{\varepsilon} w_{g,n_1}\phi_1(a) w_{1,n_1}^*\\
    &=_{\varepsilon}\sum_{k,k'\in G}\beta_k(f_{n_1})\beta_g( \bs_{g^{-1}k}^*)\fv_{g,g^{-1}k} \bs_k \bw_k \phi_1(\fu_{g,g^{-1}k}^*a) \bw_{k'}^ *\beta_{k'}(f_{n_1})\notag\\
    &=_{4\varepsilon} \sum_{k\in G} \beta_k(f_{n_1}) \beta_g( \bs_{g^{-1}k}^*)\fv_{g,g^{-1}k} \bs_k \bw_k \phi_1(\fu_{g,g^{-1}k}^*a) \bw_{k}^*. \notag
\end{align}
We proceed to check conditions \eqref{equivarianceapprox}--\eqref{closenessapprox} of the claim. Let $C=6(1+|G|)$ hereinafter. Given $a,b\in \cF$ and $g\in G$, one has 
{\small
\begin{align*}
    &\fv_{g,h}^*\beta_g(\wf_h(a))\wf_g(b)\\
     &\ \ \, \stackrel{\eqref{eqn:twiddle}\quad}{  =_{C\varepsilon}}\sum_{k,k'\in G}\fv_{g,h}^* \beta_g\beta_k(f_{n_1})\beta_g\left(\beta_h( \bs_{h^{-1}k}^*) \fv_{h,h^{-1}k} \bs_k \bw_k \phi_1(\fu_{h,h^{-1}k}^*a) \bw_k^*\right)\\
    &\quad\quad\quad\quad\quad\quad  \beta_{k'}(f_{n_1})\beta_g( \bs_{g^{-1}k'}^*)\fv_{g,g^{-1}k'}\bs_{k'} \bw_{k'} \phi_1(\fu_{g,g^{-1}k'}^*b) \bw_{k'}^*\\
    & \ \ \ \stackrel{\ref{lem:iterativelemma:eq-a}\quad}{=_{2\varepsilon}} \sum_{k,k'\in G} \fv_{g,h}^*\beta_{gk}(f_{n_1})\beta_g\left(\beta_h( \bs_{h^{-1}k}^*) \fv_{h,h^{-1}k} \bs_k \bw_k \phi_1(\fu_{h,h^{-1}k}^*a) \bw_k^*\right) \beta_{k'}(f_{n_1})\\
    &\quad\quad\quad\quad\quad  \beta_g( \bs_{g^{-1}k'}^*)\fv_{g,g^{-1}k'} \bs_{k'} \bw_{k'}\phi_1(\fu_{g,g^{-1}k'}^*b) \bw_{k'}^*\\
    &\,\stackrel{ \ref{lem:iterativelemma:eq-a},\ref{lem:iterativelemma:eq-b}}{\quad =_{2\varepsilon}} \sum_{k\in G}\fv_{g,h}^*\beta_{gk}(f_{n_1})\beta_g\left(\beta_h( \bs_{h^{-1}k}^*)\fv_{h,h^{-1}k} \bs_k \bw_k \phi_1(\fu_{h,h^{-1}k}^*a) \bw_k^*\right) \beta_g( \bs_{k}^*)\\
    &\quad\quad\quad\quad\quad \fv_{g,k}\bs_{gk} \bw_{gk} \phi_1(\fu_{g,k}^*b) \bw_{gk}^*\\
    & \quad \stackrel{\ref{lem:iterativelemma:eq-a}\quad }{=_{2\varepsilon}} \sum_{k\in G} \beta_{gk}(f_{n_1})\fv_{g,h}^* \beta_g\left(\beta_h( \bs_{h^{-1}k}^*)\fv_{h,h^{-1}k} \bs_k \bw_k \phi_1(\fu_{h,h^{-1}k}^*a) \bw_k^*\right) \beta_g( \bs_{k}^*)\\
    &\quad\quad\quad\quad\quad \fv_{g,k}\bs_{gk} \bw_{gk} \phi_1(\fu_{g,k}^*b) \bw_{gk}^*\\
    &\;\ \,\stackrel{\eqref{lem:iterativelemma:eq-choice-bw}\quad}{=_{2\varepsilon}} \sum_{k\in G}  \beta_{gk}(f_{n_1}) \fv_{g,h}^* \beta_g\beta_h( \bs_{h^{-1}k}^*)\beta_g(\fv_{h,h^{-1}k})\beta_g\beta_k\phi_1\alpha_k^{-1}(\fu_{h,h^{-1}k}^*a)\fv_{g,k}\\
    &\quad\quad\quad\quad\quad \beta_{gk}\phi_1\alpha_{gk}^{-1}(\fu_{g,k}^*b) \bs_{gk}\\
    &\quad\, \stackrel{\eqref{item:multiplicativity}}{=} \sum_{k\in G} \beta_{gk}(f_{n_1})\beta_{gh}( \bs_{h^{-1}k}^*)\fv_{g,h}^*\beta_g(\fv_{h,h^{-1}k})\fv_{g,k}\beta_{gk}\phi_1\alpha_{gk}^{-1}\big( \fu_{g,k}^* \alpha_g(\fu_{h,h^{-1}k}^*)\alpha_g(a)  b) \big)\\
    &\quad\quad\quad\quad\quad\bs_{gk} \\
    &\quad\, \stackrel{\eqref{item:anomaly}}{=} \sum_{k\in G} \beta_{gk}(f_{n_1})\beta_{gh}( \bs_{h^{-1}k}^*)\fv_{gh,h^{-1}k}\beta_{gk}\phi_1\alpha_{gk}^{-1}(\fu_{gh,h^{-1}k}^*\fu_{g,h}^*\alpha_g(a)b)) \bs_{gk}\\
    &\quad \stackrel{\eqref{lem:iterativelemma:eq-choice-bw}\ }{=_{\varepsilon}} \sum_{k\in G} \beta_{gk}(f_{n_1})\beta_{gh}( \bs_{h^{-1}k}^*)\fv_{gh,h^{-1}k} \bs_{gk} \bw_{gk}\phi_1(\fu_{gh,h^{-1}k}^*\fu_{g,h}^*\alpha_g(a)b) \bw_{gk}^* \\
    &\quad\:\, = \sum_{k\in G} \beta_{k}(f_{n_1})\beta_{gh}( \bs_{h^{-1}g^{-1} k}^*)\fv_{gh,h^{-1}g^{-1} k} \bs_{k} \bw_{k}\phi_1(\fu_{gh,h^{-1}g^{-1}k}^*\fu_{g,h}^*\alpha_g(a)b) \bw_{k}^*
\end{align*}
}
Here we have used in the penultimate step that the anomalies for $\alpha$ and $\beta$ are assumed to coincide.
If we apply \eqref{eqn:twiddle} with $gh$ in place of $g$ to the last expression, we can see that 
\[
\fv_{g,h}^*\beta_g(\wf_h(a))\wf_g(b)=_{(6|G|+21)\varepsilon} \wf_{gh}(\fu_{g,h}^*\alpha_g(a)b)
\]
for all $a,b\in \cF$ and $g, h\in G$.
Next, we compute for all $a,b\in \cF$ and $g\in G$ that

\begin{align*}
    &\wf_g(a)\wf_g(b)^*\\
    &\; \stackrel{\eqref{eqn:twiddle}\quad}{=_{C\varepsilon}} \sum_{k,k'\in G}\beta_k(f_{n_1})\beta_g( \bs_{g^{-1}k}^*)\fv_{g,g^{-1}k} \bs_k \bw_k \phi_1(\fu_{g,g^{-1}k}^*ab^*\fu_{g,g^{-1}k'}) \bw_{k'}^*\\
    &\quad\quad\quad\quad\quad \bs_{k'}^* \fv_{g,g^{-1}k'}^*\beta_g( \bs_{g^{-1}k'})\beta_{k'}(f_{n_1})\\
    &\stackrel{\ref{lem:iterativelemma:eq-a},\ref{lem:iterativelemma:eq-b}\quad}{=_{3\varepsilon}} \sum_{k\in G} \beta_g( \bs_{g^{-1}k}^*) \fv_{g,g^{-1}k}\bs_k \bw_k \phi_1(\fu_{g,g^{-1}k}^*a)\beta_k(f_{n_1})\phi_1(b^*\fu_{g,g^{-1}k}) \bw_{k}^* \bs_k^* \fv_{g,g^{-1}k}^*\\
    &\quad\quad\quad\quad\quad\beta_g( \bs_{g^{-1}k})\\
    &\ \ \; \stackrel{\ref{lem:iterativelemma:eq-a}\ }{=_{\varepsilon}} \sum_{k\in G} \beta_g( \bs_{g^{-1}k}^*)\fv_{g,g^{-1}k}\beta_k(f_{n_1}) \bs_k \bw_k \phi_1(\fu_{g,g^{-1}k}^*ab^*\fu_{g,g^{-1}k}) \bw_{k}^* \bs_k^* \fv_{g,g^{-1}k}^* \beta_g(\bs_{g^{-1}k})\\
    &\ \ \stackrel{\eqref{lem:iterativelemma:eq-choice-bw}\ }{=_{\varepsilon}} \sum_{k\in G} \beta_g( \bs_{g^{-1}k}^*)\fv_{g,g^{-1}k}\beta_{k}(f_{n_1}\phi_1\alpha_k^{-1}(\fu_{g,g^{-1}k}^*ab^*\fu_{g,g^{-1}k}))\fv_{g,g^{-1}k}^*\beta_g( \bs_{g^{-1}k})\\
    &\quad \stackrel{\eqref{item:multiplicativity}}{=} \sum_{k\in G} \beta_g( \bs_{g^{-1}k}^*\beta_{g^{-1}k}(f_{n_1}))\beta_{g}\beta_{g^{-1}k}\phi_1\alpha_{g^{-1}k}^{-1}\alpha_g^{-1}(ab^*)\beta_g( \bs_{g^{-1}k})\\
    &\quad \stackrel{\ref{lem:iterativelemma:eq-a}\ }{=_{\varepsilon}} \sum_{k\in G}\beta_g\left(\beta_k(f_{n_1}) \bs_k^*\beta_k\phi_1\alpha_k^{-1}(\alpha_g^{-1}(ab^*)) \bs_k\right)\\
    &\ \ \, \stackrel{\eqref{lem:iterativelemma:eq-choice-bw}\ }{=_{\varepsilon}} \sum_{k\in G}\beta_g\left(\beta_k(f_{n_1}) \bw_k \phi_1\alpha_{g}^{-1}(ab^*) \bw_k^* \right)\\
    &\ \ \, \stackrel{\eqref{eqn:twiddle}\quad }{=_{6\varepsilon}}\beta_g\wf_1\alpha_g^{-1}(ab^*).
\end{align*}
Here we have used \eqref{eqn:twiddle} in the last step with $1_G$ in place of $g$. Before we check condition \eqref{closenessapprox}, if we let $x_k=\beta_g( \bs_{g^{-1}k}^*)\fv_{g,g^{-1}k} \bs_k \bw_k \phi_1(\fu_{g,g^{-1}k}^*\alpha_g(a)bc)\bw_{k}^*-\phi_g(\alpha_g(a)bc)$ for $a,b,c\in \cF$ and $k,g\in G$, then we have that
\begin{align}
    \Big\lVert \sum_{k\in G}\beta_k(f_{n_1})ex_k \Big\rVert^2 &= \Big\lVert\sum_{k,k'\in G}\beta_k(f_{n_1})ex_kx_{k'}^*e\beta_{k'}(f_{n_1})\Big\rVert \label{eqn:iterativelem_orth}\\
    &\stackrel{\ref{lem:iterativelemma:eq-a}}{\leq} 16\varepsilon+\Big\lVert \sum_{k,k'\in G} ex_kx_{k'}^*e\beta_k(f_{n_1})\beta_{k'}(f_{n_1})\Big\rVert\notag\\
    &\stackrel{\ref{lem:iterativelemma:eq-b}}{\leq} 20\varepsilon+\Big\lVert\sum_{k\in G}ex_kx_k^*e\beta_k(f_{n_1})^2\Big\rVert\notag\\
    &\stackrel{\ref{lem:iterativelemma:eq-b}}{\leq} 36\varepsilon+\Big\lVert \sum_{k\in G}\beta_k(f_{n_1})ex_kx_k^*e\beta_k(f_{n_1})\Big\rVert\notag\\
    &\stackrel{\ref{lem:iterativelemma:eq-c}}{\leq} 36\varepsilon+\max_{k\in G}\lVert x_k\rVert^2 \Big\lVert \sum_{k\in G}\beta_k(f_{n_1})e\Big\rVert\notag\\
    &\leq 38\varepsilon +\max_{k\in G}\lVert x_k\rVert^2.\notag
\end{align}
\noindent Thus we also have that for $a,b,c\in \cF$ and $g\in G$
{\small
\begin{align*}
    &\left\lVert\wf_g(\alpha_g(a)bc)-\phi_g(\alpha_g(a)bc)\right\rVert\\
    &\stackrel{\eqref{eqn:twiddle}}{\leq} \Big\lVert \sum_{k\in G} \beta_k(f_{n_1})\beta_g( \bs_{g^{-1}k}^*) \fv_{g,g^{-1}k} \bs_k \bw_k \phi_1(\fu_{g,g^{-1}k}^*\alpha_g(a)bc) \bw_{k}^*-\phi_g(\alpha_g(a)bc)\Big\rVert+6\varepsilon\\
    &\, \stackrel{\ref{lem:iterativelemma:eq-c}}{\leq}\Big\lVert \sum_{k\in G} \beta_k(f_{n_1})\left(\beta_g( \bs_{g^{-1}k}^*)\fv_{g,g^{-1}k} \bs_k \bw_k \phi_1(\fu_{g,g^{-1}k}^*\alpha_g(a)bc) \bw_{k}^*-\phi_g(\alpha_g(a)bc)\right)\Big\rVert\\
    &\quad\quad+ 7\varepsilon\\
    &\stackrel{\eqref{lem:iterativelemma:eq-1}}{\leq} \Big\lVert \sum_{k\in G}\beta_k(f_{n_1})e\left(\beta_g( \bs_{g^{-1}k}^*)\fv_{g,g^{-1}k} \bs_k \bw_k \phi_1(\fu_{g,g^{-1}k}^*\alpha_g(a)bc) \bw_{k}^*-\phi_g(\alpha_g(a)bc)\right)\Big\rVert\\
    &\quad\quad +8\varepsilon\\
    &\stackrel{\eqref{eqn:iterativelem_orth}}{\leq} \max_{k\in G} \lVert \beta_g( \bs_{g^{-1}k}^*)\fv_{g,g^{-1}k} \bs_k \bw_k \phi_1(\fu_{g,g^{-1}k}^*\alpha_g(a)bc) \bw_{k}^*-\phi_g(\alpha_g(a)bc)\rVert+8\varepsilon+\sqrt{38\varepsilon}\\
    &\  \leq \max_{k\in G} \lVert\beta_g( \bs_{g^{-1}k}^*)\fv_{g,g^{-1}k} \bs_k \bw_k \phi_1(\fu_{g,g^{-1}k}^*\alpha_g(a)bc) \bw_{k}^*-\beta_g(\phi_1(a))\phi_g(b)\phi_1(c)\rVert\\
    &\quad\quad +\lVert\phi_g(\alpha_g(a)b)-\beta_g(\phi_1(a))b\rVert+8\varepsilon+\sqrt{38\varepsilon}\\
    &\stackrel{\eqref{lem:iterativelemma:eq-choice-bw}}{\leq} \max_{k\in G}\lVert \beta_g( \bs_{g^{-1}k}^*)\fv_{g,g^{-1}k}\beta_k\phi_1\alpha_k^{-1}(\fu_{g,g^{-1}k}^*\alpha_g(a)bc) \bs_k -\beta_g(\phi_1(a))\phi_g(b)\phi_1(c)\rVert  \\
    &\quad\quad +\lVert\phi_g(\alpha_g(a)b)-\beta_g(\phi_1(a))b\rVert+9\varepsilon+\sqrt{38\varepsilon}   \\
    &\ = \max_{k\in G} \lVert\beta_k\phi_1\alpha_k^{-1}(\fu_{g,g^{-1}k}^*\alpha_g(a)bc)-\fv_{g,g^{-1}k}^*\beta_g( \bs_{g^{-1}k}\phi_1(a))\phi_g(b)\phi_1(c) \bs_k^*\rVert\\
    &\quad\quad +\lVert\phi_g(\alpha_g(a)b)-\beta_g(\phi_1(a))b\rVert+9\varepsilon+\sqrt{38\varepsilon}  \\
    &\stackrel{\eqref{lem:iterativelemma:eq-1}}{\leq} \max_{k\in G} \lVert\beta_k\phi_1\alpha_k^{-1}(\fu_{g,g^{-1}k}^*\alpha_g(a)bc)-\fv_{g,g^{-1}k}^*\beta_g(\phi_{g^{-1}k}(a))\phi_g(b)\phi_k(c^*)^*\rVert\\
    &\quad\quad +\lVert\phi_g(\alpha_g(a)b)-\beta_g(\phi_1(a))b\rVert+11\varepsilon+\sqrt{38\varepsilon}  \\
    &\ \leq 11\varepsilon+\sqrt{38\varepsilon}+ \max_{k\in G} \Big( \|\beta_k\phi_1\alpha_k^{-1}(\fu_{g,g^{-1}k}^*\alpha_g(a)bc)-\phi_k(\fu_{g,g^{-1}k}^*\alpha_g(a)b)\phi_k(c^*)^*\rVert\\
    &\quad\quad +\lVert\phi_k(\fu_{g,g^{-1}k}^*\alpha_g(a)b)-\fv_{g,g^{-1}k}^*\beta_g(\phi_{g^{-1}k}(a))\phi_g(b)\rVert \Big)\\
    &\quad\quad + \lVert\phi_g(\alpha_g(a)b)-\beta_g(\phi_1(a))b\rVert.
\end{align*}
}
As $\varepsilon$ is arbitrary the result follows.
\end{proof}

We now prove Theorem \ref{thm: existence}.

\begin{proof}[Proof of Theorem \ref{thm: existence}]
Let $F_n$ be an increasing sequence of finite subsets of $A_{1}$ such that $\overline{\cup_{n\geq 1}F_n}=A_1$. Let $e_n$ be an approximate unit for $A$ such that
\[
\lVert\alpha_g(e_n)a-a\rVert<2^{-n},\quad \lVert a\alpha_g(e_n)-a\rVert<2^{-n},\quad a\in F_n,\ g\in G.
\]
Let
\[
G_n=\bigcup_{g,h\in G}\bigcup_{k\leq n} \fu_{g,h}^*F_n e_k.
\]
We shall apply Lemma \ref{lem:iterativelemma} inductively.
In the first step $n=1$ of the induction, we use Lemma \ref{lem:iterativelemma} with the input data $\varphi$, $\bv_g=1$, $\varepsilon=1/2$, $\cF=G_1$ to get unitaries $\bv_g^{(1)}$ for $g\in G$ such that letting $\phi_g^{(1)}(a)=\bv_g^{(1)}\varphi(a){\bv^{(1)}_1}^*$ for $a\in A$ and $g\in G$, we have that 
\begin{enumerate}
    \item $\lVert \fv_{g,h}^*\beta_g(\phi_h^{(1)}(a))\phi_g^{(1)}(b)-\phi_{gh}^{(1)}(\fu_{g,h}^*\alpha_g(a)b)\rVert\leq 1/2$,
    \item $\lVert\phi_g^{(1)}(a) \phi_g^{(1)}(b)^*-\beta_g\phi_1^{(1)}\alpha_g^{-1}(ab^*)\rVert\leq 1/2$,
\end{enumerate}
for $a,b\in G_1$ and $g,h\in G$. 
As the induction step, suppose $n\geq 1$ is given and for every $1\leq j \leq n$ we have a unitary family $\bv_{g}^{(j)}$ for $g\in G$ and linear maps $\phi_g^{(j)}:A\rightarrow B$ for $g\in G$, satisfying the inductive formula $\phi_g^{(j)}(a)=\bv_g^{(j)}\phi_1^{(j-1)}(a){\bv_1^{(j)}}^*$,
such that for all $g,h \in G$ and $1\leq j \leq n$ we have
\begin{enumerate}[leftmargin=*,label=\textup{(\roman*)}]
    \item \label{item:isometricstrict} $\phi_g^{(j)}(c)^*\phi_g^{(j)}(d)=\phi_1^{(j)}(c^*d)$ for all $c,d\in A$,
    \item \label{item:equivariantapprox} $\lVert \fv_{g,h}^*\beta_g(\phi_h^{(j)}(a))\phi_g^{(j)}(b)-\phi_{gh}^{(j)}(\fu_{g,h}^*\alpha_g(a)b)\rVert\leq  2^{-j}$ for all $a,b \in G_j$,
    \item \label{item:leftinner} $\lVert\phi_g^{(j)}(a) \phi_g^{(j)}(b)^*-\beta_g\phi_1^{(j)}\alpha_g^{-1}(ab^*)\rVert\leq 2^{-j}$ for all $a,b\in G_j$.
\end{enumerate}
Then we apply Lemma \ref{lem:iterativelemma} with the data $\phi^{(n)}_1, \bv_g^{(n)}, \cF=G_{n+1}, \varepsilon=2^{-(n+1)}$ to get a unitary $\bv_g^{(n+1)}$ such that setting $\phi^{(n+1)}_g(a)=\bv_g^{(n+1)}\phi_1^{(n)}(a){\bv_1^{(n+1)}}^*$ for $g\in G$ and $a\in A$ we have that $\phi_g^{(n+1)}$ is a family of linear maps satisfying \ref{item:isometricstrict}--\ref{item:leftinner}, as well as
\begin{equation} \label{item:closeness}
\begin{multlined}
	\begin{array}{cl}
	\multicolumn{2}{l}{ \lVert \phi_g^{(n+1)}(\alpha_g(e_{n})ae_{n})-\phi_g^{(n)}(\alpha_g(e_{n})ae_{n})\rVert } \\
    \leq & 2^{-(n+1)}+\lVert\beta_g\phi_1^{(n)}(e_n)\phi_g(a)-\phi_g^{(n)}(\alpha_g(e_n)a)\rVert\\
    +&\displaystyle \max_{k\in G} \Big( \lVert\fv_{g,g^{-1}k}^*\beta_g(\phi_{g^{-1}k}^{(n)}(e_n))\phi_g^{(n)}(a)-\phi_k^{(n)}(\fu_{g,g^{-1}k}^*\alpha_g(e_n)a)\rVert\\
    +&\lVert\beta_k\phi_1^{(n)}\alpha_k^{-1}(\fu_{g,g^{-1}k}^*\alpha_g(e_n)ae_n)-\phi_k^{(n)}(\fu_{g,g^{-1}k}^*\alpha_g(e_n)a)\phi_k^{(n)}(e_n)^*\rVert \Big)
    \end{array}
\end{multlined}
\end{equation}
for all $a\in G_{n+1}$, $g,h\in G$.
In particular by the inductive construction one can bound the right hand side of \eqref{item:closeness} to get that
\begin{equation}
    \lVert \phi_g^{(n+1)}(\alpha_g(e_{n})ae_{n})-\phi_g^{(n)}(\alpha_g(e_{n})ae_{n})\rVert\leq \frac{13}{2^{n+1}}
\end{equation}
for all $n\geq 1$ and $a\in F_n$.
Thus we have that $\phi^{(n)}_g(x)$ satisfies the Cauchy criterion for every $g\in G$ and $x\in \bigcup_{n\geq 1}\bigcup_{g\in G} \alpha_g(e_n) F_n e_n $.
Since the latter is a dense subset and each $\phi^{(n)}_g$ is a contractive linear map, we have that $\phi^{(n)}_g(a)$ must satisfy the Cauchy criterion for all $g\in G$ and $a\in A$.
Set
\[
\phi_g(a)=\lim_{n\rightarrow \infty}\phi_g^{(n)}(a)
\]
for $g\in G$ and $a\in A$.
Then conditions \ref{item:isometricstrict} and \ref{item:equivariantapprox} imply that $\phi=(\phi_g)_{g\in G}$ defines a $G$-coherent morphism.
Furthermore, we can read off the construction that $\phi_1$ is proper approximately unitarily equivalent to $\varphi$ through the sequence of unitaries $n\mapsto \bv_1^{(n)}\ldots \bv_1^{(2)}\bv_1^{(1)}\in\cU(\eins+B)$.
This finishes the proof.
\end{proof}
In the upcoming result we denote the stabilisation of an anomalous action $(A,\alpha,\fu)$ by $(A^s,\alpha^s,\fu^s)$, i.e., $A^s=A\otimes \bK$, $\alpha_g^s=\alpha_g\otimes \id_{\bK}$ and $\fu_{g,h}^s=\fu_{g,h}\otimes\eins$ for all $g,h\in G$.
\begin{corollary} \label{cor:unital-existence}
Let $(A,\alpha,\fu)$ and $(B,\beta,\fv)$ be anomalous actions of a finite group $G$ on separable, unital C$^*$-algebras such that $(B,\beta,\fv)$ has the Rokhlin property.
Let $\varphi:A\rightarrow B$ be a unital $*$-homomorphism such that $\varphi\alpha_g\approx_{\mathrm{u}} \beta_g \varphi$ for all $g\in G$ and $o(\alpha,\fu)=o(\beta,\fv)$.
Then there exists a unital $G$-coherent morphism $\phi:(A,\alpha,\fu)\rightarrow (B,\beta,\fv)$ such that $\phi_1\approx_{\mathrm{u}}\varphi$. 
\end{corollary}
\begin{proof}
By Theorem \ref{thm: existence} applied to the $*$-homomorphism $\varphi\otimes \id_{\bK}:A^s\rightarrow B^s$ and the stabilised actions $(A^s,\alpha^s,\fu^s)$ and $(B^s,\beta^s,\fv^s)$ there is a $G$-coherent morphism $\Phi:(A^s,\alpha^s,\fu^s)\rightarrow (B^s,\beta^s,\fv^s)$ such that $\Phi_1\approx_{\mathrm{pu}}\varphi\otimes \id_{\bK}$. Now, let $e_{11}$ be the rank one projection in $\bK$. Then there is a unitary $u\in \cU(\eins+B^s)$ such that $\Ad(u)\Phi_1(\eins_A\otimes e_{11})=\eins_B\otimes e_{11}$. In particular, the cocycle morphism $\Ad(u)\Phi$ sends $A\otimes e_{11}$ to $B\otimes e_{11}$ and we have a cocycle morphism $\phi:=\Ad(u)\Phi|_{A\otimes e_{11}}:(A,\alpha,\fu)\rightarrow (B,\beta,\fv)$. Let $(v_n)_{n\in \bN}\in \cU(\eins +B^s)$ be such that
\begin{equation}\label{eqn:corexistence}
\lim_{n\to\infty}\lVert \Ad(u)\Phi_1(a)-v_n(\varphi\otimes \id_{\bK}(a))v_n^*\rVert=0.
\end{equation}
Then the sequence $w_n=(\eins_B\otimes e_{11})v_n(\eins_B\otimes e_{11})$ in $B\otimes e_{11}$ satisfies
\[
\lim_{n\rightarrow \infty}\lVert w_n^*w_n-(\eins_B\otimes e_{11})\lVert=0=\lim_{n\rightarrow \infty}\lVert w_nw_n^*-(\eins_B\otimes e_{11})\rVert
\]
as a consequence of \eqref{eqn:corexistence} evaluated at $a=\eins_A\otimes e_{11}$. Therefore, there is a sequence of unitaries $s_n\in \cU(B)$ such that 
\[
\lim_{n\rightarrow \infty}\lVert (s_n\otimes e_{11})-w_n\rVert=0.
\]
The sequence $s_n$ witnesses the equivalence $\phi_1\approx_{\mathrm{u}}\varphi$.
\end{proof}

\section{Classification}\label{sec:classification}

We are now ready to prove classification results for anomalous actions. The following result is the analogue of \cite[Theorem 3.5]{IZU04I} in the generality of anomalous actions.

\begin{theorem}\label{thm:classicalint}
Let $G$ be a finite group and $A$ be a separable \cstar-algebra that is either unital, stable or has almost stable rank one.
Let $(\alpha,\fu)$ and $(\beta,\fv)$ be two anomalous actions of $G$ on $A$ with the Rokhlin property and $o(\alpha,\fu)=o(\beta,\fv)$.
One has that $\alpha_g\approx_{\mathrm{pu}}\beta_g$ for all $g\in G$ if and only if there is a cocycle conjugacy $\Phi=(\Phi_g)_{g\in G} :(A,\alpha,\fu)\rightarrow (A,\beta,\fv)$ with $\Phi_1\approx_{\mathrm{pu}}\id_A$.
\end{theorem}
\begin{proof}
The ``if'' part is conceptually clear, but contains one subtle point:\ If $w\in\cU\cM(A)$ is a multiplier unitary, then the automorphism $\Ad(w)$ is properly approximately inner.
This is indeed trivial when $A$ is unital.
When $A$ is stable or has almost stable rank one, this can be seen as a consequence of Lemma~\ref{lem:sequencesettingLemma}, upon multiplying $w$ with a sequential approximate unit of $A$ and viewing this as a sequence of contractions on which to apply the statement.

We shall argue the ``only if'' part.
Assume first that $A$ is either stable or has almost stable rank one.
By applying Theorem \ref{thm: existence}, with $\varphi=\id_A$, we can choose $G$-coherent morphisms $\phi:(A,\alpha,\fu)\rightarrow (A,\beta,\fv)$ and $\psi:(A,\beta,\fv)\rightarrow (A,\alpha,\fu)$ with $\phi_1\approx_{\mathrm{pu}}\id_A$ and $\psi_1\approx_{\mathrm{pu}}\id_A$.
The composition of these $G$-coherent morphisms thus satisfies $(\phi\circ\psi)_1\approx_{\mathrm{pu}} \id_A\approx_{\mathrm{pu}}(\psi\circ\phi)_1$.
Hence, as a consequence of Corollary \ref{cor: uniqueness} one has that $\phi\circ\psi\approx_{\mathrm{pu}}^G \id_A\approx_{\mathrm{pu}}^G \psi\circ\phi$.
It now follows from Theorem \ref{thm: Elliott intertwining} that there is a coycle conjugacy $\Phi:(A,\alpha,\fu)\rightarrow (A,\beta,\fv)$ with $\Phi_1\approx_{\mathrm{pu}}\phi_1\approx_{\mathrm{pu}} \id_A$.

The proof of the unital case is completely analogous, except that we use Corollary \ref{cor:unital-existence} in place of Theorem \ref{thm: existence}.
\end{proof}

With the existence and uniqueness results at hand we can now follow a similar strategy to \cite[Section 3.2]{GASA16}.
Any two anomalous actions $(A,\alpha,\fu)$ and $(B,\beta,\fv)$ of a finite group $G$ equip $\underline{K}T_u(A)$ and $\underline{K}T_u(B)$ with the structure of a $G$-module as $\underline{K}T_u$ is a functor that is invariant under inner automorphisms.
We will say that these $G$-modules are conjugate if there exists an isomorphism $\Phi:\underline{K}T_u(A)\rightarrow \underline{K}T_u(B)$ such that $\Phi\underline{K}T_u(\alpha_g)=\underline{K}T_u(\beta_g)\Phi$ for all $g\in G$.

\begin{theorem}\label{thm: classificationKTu}
Let $A$ and $B$ be two unital, simple, separable, nuclear, $\Z$-stable C$^*$-algebras in the UCT class and $G$ be a finite group.
Let $(A,\alpha,\fu)$ and $(B,\beta,\fv)$ be anomalous $G$-actions with the Rokhlin property such that $o(\alpha,\fu)=o(\beta,\fv)$.
Then every conjugacy $\kappa: \big( \underline{K}T_u(A), \underline{K}T_u(\alpha) \big) \to \big( \underline{K}T_u(B), \underline{K}T_u(\beta) \big)$ lifts to a cocycle conjugacy $\Phi: (A,\alpha,\fu)\to (B,\beta,\fv)$.
\end{theorem}
\begin{proof}
By Theorem \ref{thm:KTU}\ref{existenceKTU} there exists an isomorphism $\varphi:A\rightarrow B$ such that $\underline{K}T_u(\varphi)=\kappa$.
Hence by Theorem \ref{thm:KTU}\ref{uniquenessKTU} one has that $\varphi\alpha_g\varphi^{-1}\approx_{\mathrm{u}}\beta_g$ for all $g\in G$.
It follows from Theorem \ref{thm:classicalint} that there is a cocycle conjugacy $\Phi': (B,\varphi\alpha\varphi^{-1},\varphi(\fu))\to (B,\beta,\fv)$ such that $\Phi_1'$ is approximately inner.
Then the composition $\Phi=\Phi'\varphi$ defines a cocycle conjugacy from $(A,\alpha,\fu)$ to $(B,\beta,\fv)$ with $\underline{K}T_u(\Phi_1)=\kappa$.
Here we note that we view $\varphi=(\varphi)_{g\in G}$ as a genuine equivariant map from $(A,\alpha,\fu)$ to $(B,\varphi\alpha\varphi^{-1},\varphi(\fu))$.
\end{proof}

\begin{remark}
Theorem \ref{thm: classificationKTu} generalises \cite[Corollary 5.5]{GP25} by removing the assumption of UHF-stability.
The latter assumption is quite strong and thus the proof of Theorem \ref{thm: classificationKTu} requires rather different techniques to the more ad-hoc treatment in \cite[Corollary 5.5]{GP25}.
\end{remark}

\begin{corollary}\label{cor:classkirchbergtotalK}
Let $G$ be a finite group and let $A$ and $B$ be two separable C$^*$-algebras in the UCT class.
Let $(A,\alpha,\fu)$ and $(B,\beta,\fv)$ be anomalous actions of $G$ with the Rokhlin property and $o(\alpha,\fu)=o(\beta,\fv)$.
	\begin{enumerate}[leftmargin=*,label=\textup{(\roman*)}]
	\item Suppose $A$ and $B$ are stable Kirchberg algebras.
	Then $(A,\alpha,\fu)$ and $(B,\beta,\fv)$ are cocycle conjugate if and only if $\underline{K}(\alpha)$ is conjugate to $\underline{K}(\beta)$.\label{item:cor4.4stable}
	\item Suppose $A$ and $B$ are unital, simple, nuclear, $\cZ$-stable C$^*$-algebras of real rank zero. Then $(A,\alpha,\fu)$ and $(B,\beta,\fv)$ are cocycle conjugate if and only if $\underline{K}(\alpha)$ is conjugate to $\underline{K}(\beta)$ in a way that preserves the order and the class of the unit in $K_0$. \label{item:cor4.4unital}
	\end{enumerate}
\end{corollary}
\begin{proof}
\ref{item:cor4.4stable} follows in the same way as in the proof of Theorem \ref{thm: classificationKTu}, except that one applies Theorem \ref{thm:classificationKirchberg}\ref{item:classKirchbergstable} in place of Theorem \ref{thm:KTU}.

\ref{item:cor4.4unital}: Under the standing hypothesis $A$ and $B$ are either purely infinite or stably finite as a consequence of Kirchberg's dichotomy \cite[Theorem 4.1.10]{RO02}.
If $A$ and $B$ are purely infinite, then the result follows as in the proof of Theorem \ref{thm: classificationKTu} except that one applies Theorem \ref{thm:classificationKirchberg}\ref{item:classKirchbergunital} in place of Theorem \ref{thm:KTU}.
Suppose $A$ and $B$ are stably finite.
By combining \cite[Corollary 6.1]{TWW17} and \cite[Theorem B]{CETWW21}, we see that $A$ and $B$ have finite decomposition rank.
As they have real rank zero, it follows from \cite[Theorem 2.1]{WI07} that $A$ and $B$ are TAF.
This case now follows as in the proof of Theorem \ref{thm: classificationKTu} except that one applies Theorem \ref{thm:TAF} in place of Theorem \ref{thm:KTU}.
\end{proof}
\begin{remark}\label{rmk:observeKirchberg}
    Note that Corollary \ref{cor:classkirchbergtotalK} covers all Kirchberg algebras. Indeed, by Zhang's dichotomy \cite{ZH92} every Kirchberg algebra is either unital or stable. Moreover, every Kirchberg algebra has real rank zero.
\end{remark}
In the next sections we will show that the induced $G$-module structure of a Rokhlin anomalous action on total $K$-theory is completely determined by the $G$-module structures on ordinary $K$-theory.
This will imply that the cocycle conjugacy classes of anomalous actions on Kirchberg algebras and unital, simple, separable, nuclear TAF-algebras in the UCT class are classified by the induced $G$-module structure on the K-theory groups.
\section{K-theoretic implications of the Rokhlin property}\label{sec:cohomologyvanish}

\subsection{Cohomology vanishing}

Before we start showing the K-theoretic implications of admitting a Rokhlin anomalous action, we recall some facts about Tate cohomology.
Let $G$ be a finite group.
For a $G$-module $M$ we denote the $G$-invariants by
\[
M^G=\{x\in M: gx=x\ \forall g\in G\}
\]
and the $G$-coinvariants by
\[
M_G=M/\langle gx-x:g\in G, x\in M\rangle.
\]
Let $N:M\rightarrow M^G$ be the norm map, which is defined by $Nx=\sum_{g\in G}gx$ for $x\in M$.
Note that $N$ contains $\langle gx-x: g\in G, x\in M\rangle$ in its kernel and hence induces a well defined map $\overline{N}: M_G\rightarrow M^G$.
Recall that \emph{Tate-cohomology} groups with coefficients in $M$ is a family of groups indexed by $\bZ$ that glue together the cohomology groups and the homology groups of $G$ with coefficients in $M$ (see e.g.\ \cite[Chapter VI]{BRO82}).
\begin{definition}
The $n$-th \emph{Tate cohomology group of $G$} with coefficients in $M$ is defined by 
    \begin{equation*}
    \hat{H}^{n}(G,M):=\begin{cases}
        H^n(G,M),\quad n>0,\\
        \coker \overline{N},\quad n=0,\\
        \ker \overline{N},\quad n=-1,\\
        H_{-n-1}(G,M),\quad n<-1.
    \end{cases}
\end{equation*}
\end{definition}

We will be interested in the following notions.

\begin{definition}\label{def:cohtrivial}
A $\bZ G$-module $M$ is said to be \emph{cohomologically trivial} if $\hat{H}^n(K,M)$ vanishes for all $n\in \bZ$ and every subgroup $K\leq G$.
\end{definition}

\begin{definition}[{\cite[Definition 3.8]{IZU04II}}]
A $G$-module $M$ is called \emph{completely cohomologically trivial} if $M$ is cohomologically trivial, and for every $n\in \bN$ one of (or equivalently all of) the $G$ modules $M_n$, $nM$ or ${}_nM$ are cohomologically trivial. 
\end{definition}

Due to Izumi's prior work \cite{IZU04II} on this notion, $G$-modules of this type are structurally well-understood.

\begin{theorem}[{\cite[Theorem 3.15, Remark 3.16]{IZU04II}}]\label{theorem:inducedmodules}
Every completely cohomologically trivial $G$-module is an inductive limit of induced $G$-modules.
If the module is countable, then it is the countable inductive limit of countable induced $G$-modules.
\end{theorem}

Let $\alpha:G\rightarrow \Out(A)$ be a $G$-kernel and $K_i(A)$ the $K$-groups of $A$ for $i=0,1$.
Let $\widetilde{\alpha}_g\in \Aut(A)$ be a choice of lift of $\alpha_g$ for each $g\in G$.
As $K_i$ is functorial and does not distinguish between unitarily equivalent morphisms, it follows that $g\mapsto K_i(\widetilde{\alpha}_g)\in \Aut(K_i(A))$ is a well-defined homomorphism for $i=0,1$, which does not depend on the choice of the lift $\widetilde{\alpha}$.
This yields a well-defined $G$-module structure on $K_i(A)$ induced by $\alpha$.
In \cite[Theorem 3.3]{IZU04II} it is shown that whenever $\widetilde{\alpha}$ is an action of a finite group with the Rokhlin property on a unital, simple, separable C$^*$-algebra $A$, then the $G$-modules $K_i(A)$ are completely cohomologically trivial.
We generalise this result to the case of $G$-kernels while also dropping the assumptions of unitality and simplicity on $A$.
Let $M$ be a $G$-module.
In the following proof we set $M_0=M$.
For two unitaries $u,v$ in a C$^*$-algebra $\cU(\eins+A)$ we write $u\sim_h v$ to denote that they are homotopic in $\cU(\eins+A)$.

\begin{theorem}\label{thm:generalcohomologyvanishing}
Let $G$ be a finite group and $\alpha:G\rightarrow \Out(A)$ be a $G$-kernel on a separable C$^*$-algebra $A$ with the Rokhlin property.
Then the $G$-modules $K_i(A)$ are completely cohomologically trivial for $i=0,1$.
\end{theorem}
\begin{proof}
Following the same simplifications as in \cite[Theorem 3.3]{IZU04II}, it suffices to show that $\hat{H}^n(G,{}_m K_i(A))=0$ for $n=0,-1$, $m\in \bN\cup \{0\}$ and $i=0,1$.
Moreover, denoting by $S$ the suspension functor, the $G$-kernel $S\alpha:G\rightarrow \Out(SA)$ inherits the Rokhlin property from $\alpha$ by Lemma \ref{lem:Rokhlintensor}.
Hence, it suffices for us to show that $\hat{H}^n(G,{}_m K_1(A))=0$ for $n=0,-1$ and $m\in \bN\cup \{0\}$, as the case of coefficients in ${}_mK_0(A)$ follows by replacing the $G$-kernel $\alpha$ by $S\alpha$ and using the natural isomorphism of functors between $K_1\circ S$ and $K_0$ given by Bott periodicity.
Similarly, we may assume that $A$ is stable by passing from the $G$-kernel $\alpha$ to the $G$-kernel $\alpha\otimes \id_{\bK}:G\rightarrow \Out(A\otimes \bK)$ and using that the functors $K_i(\cdot)$ and $K_i(\cdot \otimes \bK)$ are naturally isomorphic.

Let $(\alpha,\fu)$ be a lifting of $\alpha$ (note the abuse of notation). We start by showing that $\hat{H}^0(G,{}_m K_1(A))=0$ for $m\in \bN\cup \{0\}$.
Indeed, we show that the norm maps $N_m:{}_m K_1(A)\rightarrow ({}_m K_1(A))^G$ are surjective for all $m\in \bN\cup\{0\}$.
It will then follow that $\hat{H}^0(G,{}_m K_1(A))=\coker(\overline{N_m})=\coker(N_m)=0$. Let $x\in ({}_m K_1(A))^G$ and $u\in \cU(\eins+A)$ with $x=[u]$.
We have that $[\alpha_g(u)]=[u]$ and $[u^m]=[\eins]$ and thus there exist continuous paths $u_g, s: [0,1]\to \cU(\eins+A)$ for $g\in G$ and $0=t_1\leq t_2\leq\ldots \leq t_N=1$ satisfying
    \begin{align*}
    u_g(0)=\alpha_g(u),&\quad u_g(1)=u, \\
    s(0)=\eins,&\quad s(1)=u^m,\\
    \|u_g(t_j)-u_g(t_{j+1}&)\|<1/4,\\
    \|s(t_j)-s(t_{j+1}&)\|<1/4.
    \end{align*}
for all $1\leq j < N$ and $g\in G$.
Using the Rokhlin property through Lemma~\ref{lem:Rp}, let $1/4>\varepsilon>0$ and choose a positive contraction $f\in A_{1}^+$ satisfying
\begin{align*}
    \alpha_g(f)(u_g(t_j)-\eins)&=_{\varepsilon/5} (u_g(t_j)-\eins)\alpha_g(f),\\
    f(s(t_j)-\eins)&=_{\varepsilon/5} (s(t_j)-\eins)f,\\
    \alpha_g(f^2)(u_{g}(t_j)-\eins)&=_{\varepsilon/5}\alpha_g(f)(u_g(t_j)-\eins),\\
    \alpha_g(f)\alpha_h(f)(u-\eins)&=_{\varepsilon/5} 0,\\
    \sum_{k\in G}\alpha_k(f)(u-\eins)&=_{\varepsilon/5} u-\eins,
\end{align*}
for all $g\neq h\in G$ and $1\leq j\leq N$.
Note that
\begin{align*}
(\alpha_g(f)(u_g(t_j)-\eins)+\eins)^*(\alpha_g(f)(u_g(t_j)-\eins)+\eins)&=_{\varepsilon} \eins,\\
(\alpha_g(f)(u_g(t_j)-\eins)+\eins)(\alpha_g(f)(u_g(t_j)-\eins)+\eins)^*&=_{\varepsilon} \eins,\\
(f(s(t_j)-\eins)+\eins)(f(s(t_j)-\eins)+\eins)^*&=_{\varepsilon} \eins,\\
(f(s(t_j)-\eins)+\eins)^*(f(s(t_j)-\eins)+\eins)&=_{\varepsilon} \eins,
\end{align*}
for any $g\in G$ and $1\leq j\leq N$.
Thus there exists unitaries $v_g(t_j),w(t_j)\in \cU(\eins+A)$ satisfying
\begin{align}
    \|v_g(t_j)-(\alpha_g(f)(u_g(t_j)-\eins)+\eins)\|&<\varepsilon,\\
    \|w(t_j)-(f(s(t_j)-\eins)+\eins)\|&<\varepsilon.\label{eqn:2}
\end{align}
We now set $v_g=v_g(1)$.
Note that $[v_g]=[\alpha_g(v_1)]$ for all $g\in G$.
Indeed, by construction, we have by the triangle inequality that
\[
\lVert v_g(t_j)-v_g(t_{j-1})\rVert \leq 1/2+\lVert u_{g}(t_{j})-u_g(t_{j-1})\rVert<1,\quad 2\leq j\leq N,
\]
and thus
\[
v_g=v_g(1)\sim_h v_g(t_{N-1})\sim_h \ldots \sim_hv_g(t_1)=v_g(0)\sim_h \alpha_g(v_1)
\]
for all $g\in G$, where the last homotopy follows as
\begin{align*}
&\|v_g(0)-\alpha_g(v_1)\|\\
&\leq \|v_g(0)-(\alpha_g(f)(u_g(0)-\eins)+\eins)\|+\|(f(u-\eins)+\eins)-v_1\|<1.
\end{align*}
Thus we have that 
\begin{equation}\label{eqn:1}
\sum_{g\in G} g\cdot[v_1]=\sum_{g\in G}[v_g]=[v_{1}v_{g_1}\ldots v_{g_{n}}]
\end{equation}
where $1,g_1,\ldots, g_{n}$ is an enumeration of elements of $G$ with $n=|G|-1$. 
Set $C=3^{|G|}|G|$.
Now note that 
\begin{align*}
v_1v_{g_1}\ldots v_{g_n}&=_{C\varepsilon} (f(u-\eins)+\eins)(\alpha_{g_1}(f)(u-\eins)+\eins)\ldots (\alpha_{g_{n}}(f)(u-\eins)+\eins)\\
&=_{C\varepsilon}\sum_{g\in G}\alpha_g(f)(u-\eins)+\eins\\
&=_{\varepsilon} u.
\end{align*}
Also
$v_1^m=_{(2m3^m+1)\varepsilon}f(u^m-\eins)+\eins=_{\varepsilon} w(1)$.
Thus choosing $\varepsilon$ sufficiently small we have that
\[
\sum_{g\in G}g\cdot[v_1]=[u]
\]
and also
\[
v_1^m\sim_h w(1).
\]
Moreover, by \eqref{eqn:2} we have that for $1< j\leq N$
\[
\lVert w(t_j)-w(t_{j-1})\rVert \leq 1/2 +\lVert s(t_j)-s(t_{j-1})\rVert<1
\]
and thus
\[
v_1^m\sim_h w(1)\sim_h w(t_{N-1})\ldots \sim_hw(t_1)\sim_h w(0)=\eins.
\]
In particular, the displayed relations above imply that $[v_1]^m=[1]$ and $[v_1]$ is an element of ${}_m K_1(A)$ satisfying $N_m([v_1])=x$.

Next we need to show $\hat{H}^{-1}(G,{}_m K_1(A))=0$, i.e., that if $x\in {}_m K_1(A)$ satisfies $\sum_{g\in G}gx=0$ then $x\in \langle gz-z: z\in {}_m K_1(A), g\in G\rangle$. 
\par Let $(\alpha,\fu)$ be a lifting of the $G$-kernel $\alpha$ and let $1,g_1,g_2,\ldots ,g_n$ for $n=|G|-1$ be an enumeration of all elements in $G$.
Let $x\in {}_m K_1(A)$ such that $\sum_{g\in G}gx=0$ and $u\in \cU(\eins+A)$ such that $x=[u]$. By assumption $[u^m]=[\eins]=[u\alpha_{g_1}(u)\ldots\alpha_{g_n}(u)]$ so there exist continuous paths $s, U: [0,1]\to\cU(\eins+A)$ and $0=t_1\leq t_2\leq \ldots t_N= 1$ such that
\begin{align*}
U(0)=u\alpha_{g_1}(u)\ldots \alpha_{g_n}(u)&,\quad U(1)=\eins,\\
s(0)=u^m &,\quad s(1)=\eins,\\
\lVert U(t_j)-U(t_{j-1})\rVert&\leq 1/4,\\
\lVert s(t_j)-s(t_{j-1})\rVert&\leq 1/4,\\
\end{align*}
for all $1< j \leq N$.
Using the Rokhlin property via Lemma~\ref{lem:Rp}, let $1/4>\varepsilon>0$ and choose a positive contraction $f\in A_1^+$ such that
\begin{align*}
    \alpha_g(f)\fu_{s,r}^*\fu_{k,l}(\alpha_h(u)-\eins)\fu_{k,l}^*&=_{\varepsilon/5}\fu_{s,r}^*\fu_{k,l}(\alpha_h(u)-\eins)\fu_{k,l}^*\alpha_g(f),\\
    \alpha_g(f^2)(\alpha_h(u)-\eins)&=_{\varepsilon/5} \alpha_g(f)(\alpha_h(u)-\eins),\\
    \alpha_g(f)\alpha_h(f)(u-\eins)&=_{\varepsilon/5} 0,\\
    \sum_{q\in G}\alpha_q(f)(u-\eins)&=_{\varepsilon/5} u-\eins,\\
    f(U(t_j)-\eins)&=_{\varepsilon/5}(U(t_j)-\eins)f,\\
    f^2(U(t_j)-\eins)&=_{\varepsilon/5} f(U(t_j)-\eins),\\
    \alpha_g(f)(\alpha_h(s(t_j))-\eins)&=_{\varepsilon/5}(\alpha_h(s(t_j))-\eins)\alpha_g(f),\\
     \alpha_g(f)^2(\alpha_h(s(t_j))-\eins)&=_{\varepsilon/5}\alpha_g(f)(\alpha_h(s(t_j))-\eins),
\end{align*}
for all $s,r,g\neq h,k,l\in G$, $1\leq j\leq N$. 
These conditions imply that
\begin{align*}
    (\alpha_g(f)(\alpha_h(u)-\eins)+\eins)^*(\alpha_g(f)(\alpha_h(u)-\eins)+\eins)&=_{\varepsilon} \eins\\
    (\alpha_g(f)(\alpha_h(u)-\eins)+\eins)(\alpha_g(f)(\alpha_h(u)-\eins)+\eins)^*&=_{\varepsilon} \eins,\\
    (f(U(t_j)-\eins)+\eins)^*(f(U(t_j)-\eins)+\eins)&=_{\varepsilon} \eins,\\
    (f(U(t_j)-\eins)+\eins)(f(U(t_j)-\eins)+\eins)^*&=_{\varepsilon} \eins,\\
    (\alpha_g(f)(\alpha_h(s(t_j))-\eins)+\eins)^*(\alpha_g(f)(\alpha_h(s(t_j))-\eins)+\eins)&=_{\varepsilon} \eins,\\
    (\alpha_g(f)(\alpha_h(s(t_j))-\eins)+\eins)(\alpha_g(f)(\alpha_h(s(t_j))-\eins)+\eins)^*&=_{\varepsilon} \eins,
\end{align*}
and so there are unitaries $v_{g,h}, \tilde{U}(t_j),s_{g,h}(t_j)\in \cU(\eins+A)$ and for $g,h\in G$ and $1\leq j\leq N$ such that
\begin{align*}
\|v_{g,h}-(\alpha_g(f)(\alpha_h(u)-\eins)+\eins)\|&<\varepsilon,\\
\lVert \tilde{U}(t_j)-(f(U(t_j)-\eins)+\eins)\rVert &<\varepsilon,\\
\lVert s_{g,h}(t_j)-(\alpha_g(f)(\alpha_h(s(t_j))-\eins)+\eins)\rVert&<\varepsilon.
\end{align*}
As before, set $C=3^{|G|}|G|$.
Firstly note that 
\begin{align}
    \alpha_k(v_{g,h})&=_{\varepsilon/5}\alpha_k\alpha_g(f)(\alpha_{k}\alpha_h(u)-\eins)+\eins \label{approxeq1}\\
    &=\fu_{k,g}\alpha_{kg}(f)\fu_{k,g}^*\fu_{k,h}(\alpha_{kh}(u)-\eins)\fu_{k,h}^*+\eins\notag\\
    &=_{\varepsilon/5} \fu_{k,h}(\alpha_{kh}(u)-\eins)\fu_{k,h}^*\alpha_{kg}(f)+\eins\notag\\
    &=_{\varepsilon/5}\fu_{k,h}(\alpha_{kg}(f)(\alpha_{kh}(u)-\eins)+\eins)\fu_{k,h}^*\notag\\
    &=_{\varepsilon/5}\fu_{k,h}v_{kg,kh}\fu_{k,h}^*\notag
\end{align}
and
\begin{align}
&v_{1,1}v_{g_1,1}\ldots v_{g_n,1} \label{approxeq2} \\
&=_{C\varepsilon} (f(u-\eins)+\eins)(\alpha_{g_1}(f)(u-\eins)+\eins)\ldots (\alpha_{g_n}(f)(u-\eins)+\eins)\notag\\
&=_{C\varepsilon}\sum_{g\in G}\alpha_g(f)(u-\eins)+\eins\notag\\
&=_{\varepsilon} u. \notag
\end{align}
 We also have that
\begin{align}
&v_{1,1}v_{1,g_1}\ldots v_{1,g_n}\notag\\
&=_{C\varepsilon}(f(u-\eins)+\eins)(\alpha_{g_1}(f)(u-\eins)+\eins)\ldots (\alpha_{g_n}(f)(u-\eins)+\eins)  \notag\\
&=_{C\varepsilon} f(u\alpha_{g_1}(u)\ldots \alpha_{g_n}(u)-\eins)+\eins \notag\\
&=_{\varepsilon}\tilde{U}(0)\label{approxeq3}
\end{align}
and for any $g,h\in G$
\begin{align}
v_{g,h}^m&=_{m3^m\varepsilon} (\alpha_g(f)(\alpha_h(u)-\eins)+\eins)\ldots (\alpha_g(f)(\alpha_h(u)-\eins)+\eins) \label{approxeq4}\\
&=_{m3^m\varepsilon} \alpha_g(f)(\alpha_h(u^m)-\eins)+\eins \notag\\
&=_\varepsilon s_{g,h}(0). \notag
\end{align}
As a consequence of \eqref{approxeq1}-\eqref{approxeq4} it follows 
that choosing $\varepsilon$ sufficiently small yields
\[
[\alpha_k(v_{g,h})]=[v_{kg,kh}],\quad \sum_{g\in G}[v_{g,1}]=[u],\quad \sum_{g\in G}[v_{1,g}]=[\tilde{U}(0)],\]
\[
m[v_{g,h}]=[s_{g,h}(0)], 
\]
for all $g,h,k\in G$.
As 
\begin{align*}
\lVert\tilde{U}(t_j)-\tilde{U}(t_{j-1})\rVert\leq 1/2+\lVert U(t_j)-U(t_{j-1})\rVert <1
\end{align*}
we have also that 
\[
\sum_{g\in G}[v_{1,g}]=[\tilde{U}(0)]=[\tilde{U}(t_2)]=\ldots =[\tilde{U}(t_N)]=[\eins]=0\footnote{here 0 denotes the additive unit in $K_1(A)$ which is given by the class of $\eins$ in the unitisation of $A$.}
\]
and similarly
\begin{equation}\label{eqn:anihilatem}
m[v_{g,h}]=[s_{g,h}(t_N)]=[\eins]=0.
\end{equation}
Now let 
\[
x_1=\sum_{1\neq g}[v_{1,g}]=-[v_{1,1}].
\]
By the previous computations it follows that
\[
    g_1x_1-x_1=[v_{1,1}]+[v_{g_1,1}]+\sum_{g\neq 1,g_1^{-1}}[v_{g_1,g_1g}].
\]
Letting 
\[
x_k=\sum_{g\neq 1,g_1^{-1},\ldots, g_{k-1}^{-1}}[v_{g_{k-1},g_{k-1} g}]
\]
for $1<k\leq n$ one has inductively that
\[
g_kg_{k-1}^{-1}x_k=[v_{g_k,1}]+x_{k+1}.
\]
Therefore, denoting $1$ by $g_0$ we have that
\[
    \sum_{k=1}^n g_kg_{k-1}^{-1}x_k-x_k=\sum_{g\in G}[v_{g,1}]=[u].
\]
Finally, it remains to show that $mx_k=0$ for all $1\leq k\leq n$.
This is a consequence of \eqref{eqn:anihilatem}.
\end{proof}
We will now use Theorem \ref{thm:generalcohomologyvanishing} to show that the conjugacy of the K-theory modules induced by two anomalous actions with the Rokhlin property automatically implies the conjugacy of the total K-theory modules.
This follows the same strategy as in \cite{IZU04II}, but we shall briefly sketch the strategy before we proceed with the proof.

As mentioned in Remark \ref{rmk:splitting}, for a C$^*$-algebra $A$ one can always choose splittings $s^i_{A,n}\colon{}_n K_{1-i}(A)\rightarrow K(A;\bZ_n)$ of the Künneth exact sequence that are compatible with the natural morphisms introduced in Section \ref{sec:totalK}, i.e.\ that
\begin{equation*}
\kappa_{m,n}\circ s^i_{A,n}=s^i_{A,m}\circ \times \frac{n}{(m,n)}.
\end{equation*}
Using these splittings we have isomorphisms
\begin{align*}
\psi^i_{A,n}:K_i(A)_{n}\oplus {}_{n}K_{1-i}(A)&\rightarrow K_i(A;\bZ_{n})\\
    x\oplus y&\mapsto \rho^i_{A,n}(x)+s^i_{A,n}(y).
\end{align*}
And it is a straightforward computation to check that if $\Phi_i:K_i(A)\rightarrow K_i(B)$ for $i=0,1$ are homomorphisms of abelian groups, then
\[
\psi^i_{B,n}\circ \Phi_i\oplus \Phi_{1-i}\circ(\psi^i_{A,n})^{-1}\colon K_i(A;\bZ_n)\rightarrow K_i(B,\bZ_n)
\]
for $n\in \bN$ and $i=0,1$ induces a $\Lambda$-morphism from $\underline{K}(A)$ to $\underline{K}(B)$ (by making use of \eqref{diag:compatibility}).
If we can choose the splittings $s_{A,n}^i$ equivariant, then this $\Lambda$-morphism will also be an equivariant one.
This is precisely what we do in the next lemma.\footnote{In fact we only take care of the splittings for prime powers, which suffices by \cref{lem:primesimplification}.}
For $G$-modules $M$ and  $M'$ we equip $\Hom(M,M')$ with the $G$-module structure $g\cdot f(m)=gf(g^{-1}m)$ for $m\in M$, $g\in G$ and $f\in \Hom(M,M')$.

\begin{lemma}[{cf.\ \cite[Lemma 4.1]{IZU04II}}]\label{lem:equivariantsplitting}
Let $\alpha:G\rightarrow \Out(A)$ be a $G$-kernel with the Rokhlin property of a finite group $G$ on a separable C$^*$-algebra $A$.
Then for every prime $p$, $m\in \bN$ and $i=0,1$ there exist equivariant splitting maps
    \[
    s_{A,p^m}^i:{}_{p^m}K_{1-i}(A)\rightarrow K_i(A;\bZ_{p^m})
    \]
for the Künneth sequence \eqref{eqn:Kunneth} such that
    \begin{align}
    s^i_{A,p^{m+1}}\circ \iota^{1-i}_{A,p^m}&=\kappa_{p^{m+1},p^m}\circ s^i_{A,p^m}\label{eq:compatiblesplitting1}\\
    s^i_{A,p^m}\circ \times p&=\kappa_{p^m,p^{m+1}}\circ s^i_{A,p^{m+1}}\label{eq:compatiblesplitting2}
    \end{align}
where $\iota^i _{A,p^m}:{}_{p^m}K_i(A)\rightarrow {}_{p^{m+1}}K_i(A)$ is the inclusion map.
\end{lemma}
\begin{proof}
The argument is identical to that of \cite[Lemma 4.1]{IZU04II}, except that \cite[Theorem 3.3]{IZU04II} is replaced by Theorem \ref{thm:generalcohomologyvanishing}.
We nevertheless give a detailed proof to clarify how Theorem \ref{thm:generalcohomologyvanishing} enters.
By \cite{BO79,BO80} we let $t^i_{p^m}:{}_{p^m}K_{1-i}(A)\rightarrow K_i(A;\bZ_{p^m})$ be splittings of the Künneth exact sequence such that
	\begin{align*}
    \kappa_{p^{m+1},p^m}\circ t^i_{p^m}&=t^i_{p^{m+1}}\circ \iota^{1-i}_{p^m}\\
    \kappa_{p^m,p^{m+1}}\circ t^i_{p^{m+1}}&=t^i_{p^m}\circ \times p
    \end{align*}
where we have dropped $A$ from the notation of the Bockstein maps for brevity.
For the remainder of the proof fix $i\in \{0,1\}$.
We will inductively construct $r^i_{p^m}:{}_{p^m}K_{1-i}(A)\rightarrow K_i(A)_{p^m}$ for $m\geq 1$ such that
    \begin{align}
    g\cdot t^i _{p^m}-t_{p^m}^i&=g\cdot \rho^i_{p^m}\circ r^i_{p^m}-\rho^i_{p^m}\circ r_{p^m}^i,\label{eqn:inductive1}\\
    r^i_{p^{m+1}}\circ \iota^{1-i}_{p^{m}}&=\times p\circ r_{p^m}^i,\label{eqn:inductive2}\\
    \pi^i_{p^{m}}\circ r^i_{p^{m+1}}&=r^i_{p^m}\circ \times p,\label{eqn:inductive3}
    \end{align}
where $\pi^i_{p^m}:K_i(A)_{p^{m+1}}\rightarrow K_i(A)_{p^{m}}$ is the canonical quotient map. This will yield the required result.
Indeed, setting $s^i_{p^m}=t_{p^m}^i-\rho_{p^m}^i\circ r_{p^m}^i$ we have that by \eqref{eqn:inductive1} $g\cdot s^i_{p^m}=s^i_{p^m}$ and it is a straightforward computation from \eqref{diag:compatibility} that \eqref{eqn:inductive2} and \eqref{eqn:inductive3} are equivalent to
    \begin{align*}
    \kappa_{p^{m+1},p^m}\circ \rho^i_{p^m}\circ r^i_{p^m}&=\rho_{p^{m+1}}^i\circ t^i_{p^{m+1}}\circ \iota^{1-i}_{p^m}\\
    \kappa_{p^m,p^{m+1}}\circ \rho_{p^{m+1}}^i\circ r^i_{p^{m+1}}&=\rho^i _{p^m}\circ r^i_{p^m}\circ \times p
    \end{align*}
which is precisely that $\rho^i_{p^m}\circ r^i_{p^m}$ is compatible with the Bockstein operations, and thus so is $s^i_{p^m}$.
For the case $m=1$ note that the $g\cdot t_{p}^i-t_{p}^i$ defines a $1$-cocycle with values in $\Hom({}_p K_{1-i}(A),\im(\rho_p^i))$.
Thus  $(\rho_{p}^i)^{-1}(g\cdot t_{p}^i-t_{p}^i)$ is an element of $Z^1(G,\Hom({}_p K_{1-i}(A),K_i(A)_p))$.
By Theorem \ref{thm:generalcohomologyvanishing} we have that $K_i(A)_p$ is cohomology vanishing and thus by \cite[Lemma 3.12 (2), Lemma 3.11 (1)]{IZU04II} also $\Hom({}_pK_{1-i}(A),K_i(A)_p)$.
Thus every element of $Z^1(G,\Hom({}_p K_{1-i}(A),K_i(A)_p))$ is a coboundary and there exists an element $r^i_p\in\Hom({}_p K_{1-i}(A),K_i(A)_p)$ such that
    \[
    (\rho_{p}^i)^{-1}(g\cdot t_{p}^i-t_{p}^i)=\partial r_{p}^i(g)=g\cdot r^i_{p}-r^i_p
    \]
as required.
By induction let $r_{p^j}^i$ for $1\leq j\leq m$ be such that \eqref{eqn:inductive1}-\eqref{eqn:inductive3} hold.
Then, a non-equivariant argument (see e.g. the Lemma in \cite[Section 2]{BO80} or the proof of \cite[Lemma 4.1]{IZU04II}) there exists $\ti{r}_{p^{m+1}}^i\colon{}_{p^{m+1}} K_{1-i}(A)\rightarrow K_i(A)_{p^{m+1}}$ satisfying
    \begin{align*}
    \ti{r}_{p^{m+1}}^i\circ \iota_{p^m}^{1-i}&=\times p\circ r_{p^m}^i\\
    \pi^i_{p^m}\circ \ti{r}_{p^{m+1}}^i&=r^i_{p^m}\circ \times p.
    \end{align*}
I.e. $\ti{r}_{p^{m+1}}^i$ satisfies \eqref{eqn:inductive2} and \eqref{eqn:inductive3}.
Note that for any function $f\in \Hom(\coker(\iota_{p^m}^{1-i}),\ker(\pi_{p^m}^i))$, the function $r^i_{p^{m+1}}=\ti{r}^i_{p^{m+1}}+f$ will still satisfy \eqref{eqn:inductive2} and \eqref{eqn:inductive3}.
We will choose such $f$ to ensure that $r^i_{p^{m+1}}$ also satisfies \eqref{eqn:inductive1}.
Let
    \[
    c(g)=(\rho_{p^{m+1}}^i)^{-1}(g\cdot t_{p^{m+1}}^i-t_{p^{m+1}}^i)-(g\cdot \ti{r}_{p^{m+1}}^i-\ti{r}_{p^{m+1}}^i).
    \]
Then $c$ defines a $1$-cocycle valued in $\Hom({}_{p^{m+1}}K_{1-i}(A),K_i(A)_{p^{m+1}})$ that measures the failure of \eqref{eqn:inductive1}.
Moreover,
    \begin{align*}
    c(g)\iota_{p^m}^{1-i}&=(\rho_{p^{m+1}}^i)^{-1}(g\cdot t_{p^{m+1}}^i-t_{p^{m+1}}^i)\iota_{p^m}^{1-i}-(g\cdot \ti{r}_{p^{m+1}}^i-\ti{r}_{p^{m+1}}^i)\iota_{p^m}^{1-i}\\
    &=(\rho_{p^{m+1}}^i)^{-1}\kappa_{p^{m+1},p^m}(g\cdot t^i_{p^m}-t^i_{p^m})-\times p\circ (g\cdot r^i_{p^m}-r^i_{p^m})\\
    &=(\rho_{p^{m+1}}^i)^{-1}\kappa_{p^{m+1},p^m} \rho^i_{p^m}((\rho^i_{p^m})^{-1}(g\cdot t^i_{p^m}- t^i _{p^m})-g\cdot r^i_{p^m}+r^i_{p^m})\\
    &=0
    \end{align*}
by the induction hypothesis. 
Similarly $\pi^i_{p^m} c(g)=0$.
Thus $c$ induces an element of $Z^1(G,\Hom(\coker(\iota_{p^m}^{1-i}),\ker(\pi_{p^m}^i))$.
By Theorem \ref{thm:generalcohomologyvanishing} the $G$-module $p^mK_i(A)_{p^{m+1}}=\ker(\pi_{p^m}^i)$ is completely cohomologically trivial.
Thus by \cite[Lemma 3.12 (2), Lemma 3.11 (1)]{IZU04II} so is $\Hom(\coker(\iota_{p^m}^{1-i}),\im(\pi_{p^m}^i))$.
In particular, there exists $f\in \Hom(\coker(\iota_{p^m}^{1-i}),\im(\pi_{p^m}^i))$ such that
    \[
    (\rho_{p^{m+1}}^i)^{-1}(g\cdot t_{p^{m+1}}^i-t_{p^{m+1}}^i)-(g\cdot \ti{r}_{p^{m+1}}^i-\ti{r}_{p^{m+1}}^i)=\partial f(g)=g\cdot f-f.
    \]
Then $r_{p^{m+1}}^i=\ti{r}_{p^{m+1}}^i+f$ satisfies \eqref{eqn:inductive1}-\eqref{eqn:inductive3}.
\end{proof}

We may now improve Corollary \ref{cor:classkirchbergtotalK} to simplify the classifying invariant of anomalous actions with the Rokhlin property on Kirchberg algebras.

\begin{theorem}\label{thm:anomclassKtheory}
Let $G$ be a finite group and let $A$ and $B$ be two separable C$^*$-algebras in the UCT class.
Let $(A,\alpha,\fu)$ and $(B,\beta,\fv)$ be anomalous actions of $G$ with the Rokhlin property and $o(\alpha,\fu)=o(\beta,\fv)$.
	\begin{enumerate}[leftmargin=*,label=\textup{(\roman*)}]
	\item Suppose $A$ and $B$ are stable Kirchberg algebras.
	Then $(A,\alpha,\fu)$ and $(B,\beta,\fv)$ are cocycle conjugate if and only if $K_*(\alpha)$ is conjugate to $K_*(\beta)$.\label{item:thm5.7stable}
	\item Suppose $A$ and $B$ are unital, simple, nuclear, $\cZ$-stable C$^*$-algebras of real rank zero.
	Then $(A,\alpha,\fu)$ and $(B,\beta,\fv)$ are cocycle conjugate if and only if $K_*(\alpha)$ is conjugate to $K_*(\beta)$ in a way that preserves the order and the class of the unit in $K_0$.\label{item:thm5.7unital}
	\end{enumerate}
\end{theorem}
\begin{proof}
The ``only if'' parts are immediate, so we prove the ``if'' parts.

\ref{item:thm5.7stable}: Let $\Psi^i:K_i(A)\rightarrow K_i(B)$ be equivariant isomorphisms for $i=0,1$.
As $(\alpha,\fu)$ and $(\beta,\fv)$ have the Rokhlin property we may choose equivariant splittings of the Künneth formula $s^i_{A,p^m}$ and $s^i_{B,p^m}$ by Lemma \ref{lem:equivariantsplitting}.
In particular, we have equivariant isomorphisms
    \begin{align*}
    \psi^i_{C,p^m}:K_i(C)_{p^m}\oplus {}_{p^m}K_{1-i}(C)&\rightarrow K_i(C;\bZ_{p^m})\\
    x\oplus y&\mapsto \rho^i_{A,p^m}(x)+s^i_{A,p^m}(y)
    \end{align*}
for $i=0,1$, prime $p$, $m\in \bN$ and $C=A$ or $B$ with the respective actions induced by $\alpha$ and $\beta$.
Furthermore, we have equivariant isomorphisms
    \[
    \Psi^i_{p^m}=\psi^i_{B,p^m}\circ\Psi_i\oplus \Psi_{1-i}\circ(\psi^i_{A,p^m})^{-1}\colon K_i(A;\bZ_{p^m})\rightarrow K_i(B;\bZ_{p^m})
   \]
for $i=0,1$, $m\in \bN$ and prime $p$.
Moreover, as the splittings $s$ are compatible, it is a straightforward computation that the collection $\Psi^i_{p^m}$ ad $\Phi_i$ define an element in $\Hom_{\Lambda}(F_P\underline{K}(A),F_P\underline{K}(B))$ which we denote by $\Psi$.
Then using the notation of Lemma \ref{lem:primesimplification}, $\Phi_{A,B}(\Psi)\in \Hom_{\Lambda}(\underline{K}(A),\underline{K}(B))$ is an equivariant isomorphism.
Indeed, by the equivariance of $\Psi$ and the properties of $\Phi_{A,B}$ we have that $\Phi_{A,B}(\Psi)$ is invertible and for every $g\in G$ one has
    \[
    \Phi_{A,B}(\Psi)\underline{K}(\alpha_g)=\Phi_{A,B}(\Psi\circ F_P\underline{K}(\alpha_g))=\Phi_{A,B}(F_P\underline{K}(\beta_g)\circ \Psi)=\underline{K}(\beta_g)\Phi_{A,B}(\Psi).
    \]
Thus the result follows from \cref{cor:classkirchbergtotalK}\ref{item:cor4.4stable}.

\ref{item:thm5.7unital}: The proof follows in the same way as in \ref{item:thm5.7stable} but noting that if $\Psi^0$ preserves the order and the unit then so does $\Phi_{A,B}(\Psi)$, and using Corollary \ref{cor:classkirchbergtotalK}\ref{item:cor4.4unital} in place of Corollary \ref{cor:classkirchbergtotalK}\ref{item:cor4.4stable}.
\end{proof}

As a consequence of the theorem above we also classify $G$-kernels with the Rokhlin property (see \cite[Corollary 4.8]{GP25} for a similar argument).

\begin{corollary}\label{cor:Gkerclass}
Let $A$ and $B$ be separable C$^*$-algebras in the UCT class, and let $\alpha:G\rightarrow \Out(A)$, $\beta:G\rightarrow\Out(B)$ be $G$-kernels with the Rokhlin property.
If $A$ and $B$ are stable Kirchberg algebras, then $\alpha$ and $\beta$ are conjugate if and only if $K_*(\alpha)$ and $K_*(\beta)$ are conjugate and $\ob(\alpha)=\ob(\beta)$.
If $A$ and $B$ are unital, simple, nuclear, $\cZ$-stable C$^*$-algebras of real rank zero, then the same conclusion holds, with the additional requirement that the conjugacy between $K_0(\alpha)$ and $K_0(\beta)$ preserves the order and the class of the unit.
\end{corollary}

\section{Rokhlin $G$-kernels realizing $K$-theory modules}\label{sec:range}

In this section we determine the $G$-modules arising as $K$-theory groups of Kirchberg algebras that admit a Rokhlin $G$-kernel with a given lifting obstruction. 
We shall start by constructing a model action. 
This is based on the construction in \cite[Lemma 5.2]{IZU04II}. 
The case when $G$ is a finite cyclic group is shown in \cite[Proposition 6.4.7]{THESIS}.

\begin{lemma}\label{lem:Gkerconstruction}
Let $G$ be a finite group and $[\lambda] \in H^3(G,\bT)$.
Let $D_G$ be the unital Kirchberg algebra in the UCT class with $K_0(D_G)=\bZ[G]$ and $K_1(D_G)=0$.
Then there exists a $G$-kernel $\alpha:G\rightarrow \Out(D_G)$ with the Rokhlin property with $\ob(\alpha)=[\lambda]$ and such that $K_0(\alpha)$ coincides with the left-translation action.
\end{lemma}
\begin{proof}
We use the same inductive construction of $D_G$ as in \cite[Lemma 5.2]{IZU04II}.
Choose a partition of unity $\{p_g\}_{g\in G}$ of non-zero projections in $\cO_{\infty}$ such that $[p_g]=0$ for all $g\neq 1$ and $[p_1]=[\eins]$. 
For each $n\in\mathbb N$, set 
\[
D_n=C\bigg(G,\bigotimes_{k=1}^n \cO_{\infty}\bigg)
\]
and $\varphi^{(n)}:D_n\rightarrow D_{n+1}$ defined by $\varphi^{(n)}(f)(g)=\sum_{h\in G}f(gh^{-1})\otimes p_h$.
Then the inductive limit $D_G=\lim \{ D_n,\varphi^{(n)} \}$ is a Kirchberg algebra in the UCT class with $K$-theory $K_0(D_G)=\bZ[G]$ and $K_1(D_G)=0$.
It remains to construct anomalous actions $(\alpha^{(n)},\fu^{(n)})$ on $D_n$ with $o(\alpha^{(n)},\fu^{(n)})=\lambda$ and unitaries $\bv^{(n)}_g$ for $g\in G$ such that $(\varphi^{(n)},\bv^{(n)}):(\alpha^{(n)},\fu^{(n)})\rightarrow (\alpha^{(n+1)},\fu^{(n+1)})$ define cocycle morphisms.
This data will then induce an anomalous action as the inductive limit
    \[
    (\alpha,\fu)=\lim \{ (\alpha_n,\fu_n), (\varphi_n,\bu_n)\}
    \]
on $D_G$ with $o(\alpha,\fu)=\lambda$.
For each $g,h,k\in G$ and $n\in \bN$, set 
    \begin{align*}
        \alpha^{(n)}_g(f)(h)&=f(g^{-1}h)\\
        \fu^{(n)}_{g,h}(k)&=\lambda(k^{-1},g,h)
    \end{align*}
and
    \begin{equation}
        \bv_g^{(n)}(h)=\sum_{\ell\in G}\lambda(\ell,h^{-1},g)\otimes p_\ell.
    \end{equation}
We now check that $(\varphi^{(n)},\bv^{(n)}):(\alpha^{(n)},\fu^{(n)})\rightarrow (\alpha^{(n+1)},\fu^{(n+1)})$ is a cocycle morphism.
It is clear that $\varphi^{(n)}\alpha_g^{(n)}=\alpha_g^{(n+1)}\varphi^{(n)}=\Ad(\bv_g^{(n)})\alpha_g^{(n+1)}\varphi^{(n)}$.
By the cocycle identity for $\lambda$, we see
    \begin{align*}
    &\bv_g^{(n)}\alpha_g^{(n+1)}(\bv_h^{(n)})u_{g,h}^{(n+1)}\bv_{gh}^{(n)*}(k) \\
    &=\sum_{\ell\in G} \lambda(\ell,k^{-1},g)\lambda(\ell,k^{-1}g,h)\lambda(k^{-1},g,h)\overline{\lambda(\ell,k^{-1},gh})\otimes p_\ell \\
    &=\sum_{\ell\in G}\lambda(\ell k^{-1},g,h)\otimes p_\ell \\
    &=\varphi^{(n)}(\fu_{g,h}^{(n)})(k)
    \end{align*}
for all $g,h,k\in G$ and $n\in \bN$.
Moreover, the central projections $q_g^{(n)}\in D_n$ given by the delta functions at $g\in G$ satisfy $\alpha_g^{(n)}(q_h^{(n)})=q_{gh}^{(n)}$.
Thus the sequence of projections $q=(q_1^{(n)})\in D_G$ defines a Rokhlin projection in $F_\infty(D_G)$ for $(\alpha,\fu)$.
As $\alpha^{(n)}$ is given by left translation at each stage, it is also easy to see that  $K_0(D_G)=\bZ[G]$ as a $G$-module.
\end{proof}

We may now generalise \cite[Theorem 5.3]{IZU04II} to the setting of $G$-kernels.
In light of Corollary \ref{cor:Gkerclass}, one should think of the succeeding result as a range of invariant result for $G$-kernels of a given lifting obstruction with the Rokhlin property on Kirchberg algebras.

\begin{theorem}\label{thm:rangeofinvstable}
Let $G$ be a finite group, $[\lambda]\in H^3(G,\bT)$, $M_0$ and $M_1$ a pair of countable completely cohomologically trivial $G$-modules and $e\in M_0^G$.
Then:
    \begin{enumerate}[leftmargin=*,label=\textup{(\roman*)}]
    \item There exists a $G$-kernel $\alpha:G\rightarrow \Out(A)$ with the Rokhlin property on a stable Kirchberg algebra in the UCT class such that $\ob(\alpha)=[\lambda]$ and $K_i(A)\cong M_i$ for $i=0,1$ as $G$-modules.
    Moreover $\alpha$ is unique up to conjugacy. \label{item:onethmrange}
    \item There exists a $G$-kernel with the Rokhlin property on a unital Kirchberg algebra in the UCT class $A$ such that $\ob(\alpha)=[\lambda]$ and $(K_0(A),[1]_A)\cong (M_0,e)$ and $K_1(A)\cong M_1$ as $G$-modules (with distinguished element).
    Moreover $\alpha$ is unique up to conjugacy.\label{item:twothmrange}
    \end{enumerate}
\end{theorem}
\begin{proof}
The uniqueness in both cases follows from Corollary \ref{cor:Gkerclass}. We turn to the existence.
As $M_0$ and $M_1$ are completely cohomologically trivial $G$-modules, it follows from Theorem \ref{theorem:inducedmodules} that there exists a sequential inductive system of $G$-modules
\[
\phi_i^{(n)}:\bZ[G]\otimes M_i^{(n)}\rightarrow \bZ[G]\otimes M_i^{(n+1)},
\]
with $M_i^{(n)}$ countable abelian groups with trivial $G$-module structures, such that $M_i=\lim \{ \bZ[G]\otimes M_i^{(n)}, \phi_i^{(n)} \}$ for $i=0,1$.
For each $n\in\mathbb N$, let $A^{(n)}$ be the stable Kirchberg algebra in the UCT class such that $K_i(A^{(n)})=M_i^{(n)}$ for $i=0,1$.
By Lemma \ref{lem:Gkerconstruction}, there exist $\lambda$-anomalous actions with the Rokhlin property $(\alpha^{(n)},\fu^{(n)})$ on $B^{(n)}=D_G\otimes A^{(n)}$ inducing the module structure $\bZ[G]\otimes M_i^{(n)}$ in K-theory (namely the model action in the lemma tensored with the trivial action on $A^{(n)}$).
By Lemma \ref{lem:equivariantsplitting}, it is possible to extend the equivariant maps $\phi_i^{(n)}$ to equivariant maps $\Phi^{(n)}:\underline{K}(B^{(n)})\rightarrow \underline{K}(B^{(n+1)})$.
Theorem \ref{thm:classificationKirchberg}\ref{item:classKirchbergstable} allows us to lift $\Phi^{(n)}$ to a $*$-homomorphism $\varphi^{(n)}:B^{(n)}\rightarrow B^{(n+1)}$ such that $\underline{K}(\varphi^{(n)})=\Phi^{(n)}$ and so $\varphi^{(n)}\alpha^{(n)}_g\approx_{\mathrm{pu}}\alpha_g^{(n+1)}\varphi^{(n)}$.
It follows from Theorem \ref{thm: existence} that there exist $G$-coherent morphisms $\psi^{(n)}: (B^{(n)},\alpha^{(n)},\fu^{(n)}) \rightarrow (B^{(n+1)},\alpha^{(n+1)},\fu^{(n+1)})$ such that $\psi_1^{(n)}\approx_{\mathrm{pu}}\varphi^{(n)}$.
In particular $K_i(\psi_1^{(n)})=K_i(\varphi^{(n)})=\phi_i^{(n)}$ for $i=0,1$.
The inductive limit of anomalous $G$-actions 
\[
(A,\alpha,\fu)=\lim_{\longrightarrow} \{ (B^{(n)},\alpha^{(n)},\fu^{(n)}), \psi^{(n)}\}
\]
yields a $\lambda$-anomalous action on a stable Kirchberg algebra $A$ as in \ref{item:onethmrange}. 

To show \ref{item:twothmrange}, first we apply the first part and choose a stable Kirchberg algebra $B$ in the UCT class and $\alpha:G\rightarrow \Out(B)$ the Rokhlin $G$-kernel with $K_0(B)\cong M_0$ and $K_1(B)\cong M_1$ as $G$-modules and such that the lifting obstruction of $\alpha$ is $[\lambda]\in H^3(G,\bT)$.
As $B$ is a stable Kirchberg there is a projection $p\in B$ such that $[p]=e$ under the identification of $K_0(B)$ and $M_0$.
Let $\widetilde{\alpha}: G\to\Aut(B)$ be a lift for $\alpha$.
As $e$ is assumed to be a fixed point, the projection $\widetilde{\alpha}_g(p)$ has the same $K_0$-class as $p$ for every $g\in G$, so these projections are unitarily equivalent.
This implies that $\alpha$ naturally induces a $G$-kernel on $pAp$ with lifting obstruction $\lambda$ (see \cite[Lemma 3.4]{EVGI23}).
Namely, if we choose unitaries $v_g\in\cU\cM(B)$ with $p  = v_g\widetilde{\alpha}_g(p)v_g^*$ for each $g\in G$, then the map $G\to\Aut(pAp)$ given by $g\mapsto \Ad(v_g)\widetilde{\alpha}_g|_{pAp}$ defines a $G$-kernel on the corner $pAp$ with the desired properties. 
\end{proof}

\bibliographystyle{abbrv}
\bibliography{refs}
\end{document}